\documentclass[11pt]{amsart}
\usepackage{amssymb}
\usepackage{amsfonts} 
\usepackage{latexsym}   
\usepackage{amssymb}    
\usepackage{amsmath}    
\usepackage{amsbsy}
\usepackage{amsgen}
\usepackage{amsfonts}
\usepackage{array}
\usepackage{epsfig}
\usepackage{psfrag}
\usepackage[all]{xy}
\usepackage{hyperref}
\usepackage{tikz-cd}
\usepackage{xcolor}
\usepackage{amssymb}
\usepackage{mathabx}
\usepackage{amsfonts}
\usepackage{hyperref}\hypersetup{pdfborder={0 0 .8}}
\usepackage{amsmath}
\usepackage{graphicx}
\usepackage[all]{xy}%
\usepackage[margin=.8in,footskip=0.25in]{geometry}

\providecommand{\U}[1]{\protect\rule{.1in}{.1in}}

\newcommand{\R}{\mathbb{R}}

\theoremstyle{plain}

\newtheorem{algorithm}{Algorithm}[section]

\newtheorem{corollary}[algorithm]{Corollary}

\newtheorem{definition}[algorithm]{Definition}

\newtheorem{observation}[algorithm]{Observation}

\newtheorem{lemma}[algorithm]{Lemma}

\newtheorem{theorem} [algorithm] {Theorem}

\newtheorem{theoremlet'}[thm]{Theorem$'$}

\newtheorem*{conn}{Connectedness Lemma}
\newtheorem*{period}{Periodicity Lemma}

\newtheorem*{mod2codim4lemma}{$\mathbf{ 4\,(\textrm{mod}\,{8})}$ Codimension Lemma}
\newtheorem*{oddcodimlemma}{Odd Codimension Lemma}
\newtheorem*{codim2mod4}{{\bf$\mathbf{2\,(\textrm{mod}\,{4})}$\,Codimension Lemma}}

\newtheorem*{TA}{Theorem A}
\newtheorem*{TB}{Theorem B}

\usepackage{amssymb}
\usepackage{amsfonts} 
\usepackage{latexsym}   
\usepackage{amssymb}    
\usepackage{amsmath}    
\usepackage{amsbsy}
\usepackage{amsgen}
\usepackage{amsfonts}
\usepackage{array}
\usepackage{scalerel}
\usepackage{epsfig}
\usepackage{psfrag}
\usepackage[all]{xy}
\usepackage{amssymb}
\usepackage{mathabx}
\newtheorem{remark}[algorithm]{Remark}
\newtheorem{proposition}[algorithm]{Proposition}

\newtheorem*{observe*}{Observation}

\def\codim{\textrm{codim}}
\def\bdm{\begin{displaymath}}
\def\edm{\end{displaymath}}
\def\beq{\begin{equation}}
\def\eeq{\end{equation}}
\def\bes{\begin{equation*}}
\def\ees{\end{equation*}}
\def\lista{\begin{enumerate}}
\def\listb{\end{enumerate}}
\def\epcm{\end{picture}\end{center}\end{minipage}}
\def\bpcm{\begin{minipage}{80pt}\begin{center}\begin{picture}}

\def\t2{T^2}

\def\f4{F_4}
\def\g2{G_2}

\def\p2{\frac{\pi}{2}}

\def\rk{\textrm{rk}}

\def\({\left(}
\def\){\right)}
\def\dim{\textrm{dim}}

\def\<{\langle}
\def\>{\rangle}

\newcommand{\euc}{\raisebox{2pt}{$\chi$}}

 \numberwithin{equation}{section}
  \numberwithin{figure}{section}

\newcommand{\N}{\mathbb{N}}\newcommand{\Z}{\mathbb{Z}}
\renewcommand{\R}{\mathbb{R}}\newcommand{\C}{\mathbb{C}}

\newcommand{\F}{\mathbb{F}}

\newcommand{\RP}{\mathbb{R}\mathrm{P}}
\newcommand{\CP}{\mathbb{C}\mathrm{P}}
\newcommand{\HP}{\mathbb{H}\mathrm{P}}
\newcommand{\FP}{\mathbb{F}\mathrm{P}}

\renewcommand{\(}{\left(}
\renewcommand{\)}{\right)}

\newtheorem*{ack}{Acknowledgements}
\newtheorem*{org}{Organization}

\begin{document}

\newcommand{\comment}[1]{\vspace{5 mm}
\par \noindent
\marginpar{\textsc{Note}}
\framebox{\begin{minipage}[c]{0.95 \textwidth}
#1 \end{minipage}}\vspace{5 mm}\par}

\title[On Fixed-Point Sets of $\Z_2$-Tori in Positive Curvature]{On Fixed-Point Sets of $\Z_2$-Tori in Positive Curvature}

\author[Bosgraaf]{Austin Bosgraaf}
\address[Bosgraaf]{Department of Mathematics, Oregon State University, 
Corvallis, Oregon}
\email{}

\author[Escher]{Christine Escher}
\address[Escher]{Department of Mathematics, Oregon State University, Corvallis, Oregon}
\email{escherc@oregonstate.edu}

\author[Searle]{Catherine Searle}
\address[Searle]{Department of Mathematics, Statistics, and Physics, Wichita State University, Wichita, Kansas}
\email{Catherine.Searle@wichita.edu}


\date{\today}

%

\begin{abstract}

 In recent work of Kennard, Khalili Samani, and the last author, they generalize the Half-Maximal Symmetry Rank result of Wilking for torus actions on positively curved manifolds to $\mathbb{Z}_2$-tori with a fixed point. They show that if the rank is approximately one-fourth of the dimension of the manifold, then fixed point set components of small co-rank subgroups of the $\Z_2$-torus  are homotopy equivalent to spheres, real projective spaces, complex projective spaces, or lens spaces. In this paper, we lower the bound on the rank of the $\mathbb{Z}_2$-torus to approximately $n/6$ and $n/8$ and are able to classify either the integral cohomology ring or the $\mathbb{Z}_2$-cohomology ring, respectively, of the fixed point set of the $\mathbb{Z}_2$-torus.

\end{abstract}
\maketitle


\section{Introduction}

\pagecolor{white}

The classification of
positively curved manifolds is 
a long standing open problem 
in Riemannian geometry.  
In particular, in dimensions greater than 24 the only known simply connected examples are compact rank-one symmetric spaces, that is, they are one of $S^n$, $\CP^n$, or $\HP^n$. 

The Symmetry Program suggests that one approach such a classification with the additional hypothesis of an isometric group 
action. This program has been quite successful over the last 30 years, producing new tools, new techniques, and even some new examples. The case of torus actions has naturally attracted a great deal of attention, see, for example, work of Grove and the last author \cite{GS}, Rong \cite{R}, Fang and Rong \cite{FR}, Wilking \cite{W}, as well as more recent work of Kennard, Wiemeler, and Wilking \cite{KWW1, KWW2}. It is a natural next step  to study discrete abelian actions. Previous work on positively and non-negatively curved manifolds of dimension $4$ with discrete symmetries  can be found in Yang \cite{Yan94}, Hicks \cite{Hic97}, Fang \cite{Fan08}, and Kim and Lee \cite{KL09}, and  for manifolds of higher dimensions in Fang and Rong \cite{FR},  Su and Wang \cite{SW}, Wang \cite{W}, and most recently in Kennard, Khalili Samani, and the last author \cite{KKSS}.

In this paper, we focus on 
actions of $\Z_2$-tori, that is, $\Z_2^r$-actions. As in the work in \cite{KKSS}, our results require the assumption that the $\Z_2$-torus has a fixed point. 
In Theorem C of \cite{KKSS}, they show that for a closed, positively curved manifold with an effective and isometric  $\Z_2^r$-action on $M^n$ with a fixed point,  
if $n\geq 15$ and $r\geq \frac{n+3}{4} +1$, then for any subgroup $\Z_2^r$ with corank at most four, the fixed-point set component $F^m$ at $x$ is homotopy equivalent to $S^m$, $\RP^m$, $\CP^\frac{m}{2}$, or a lens space; and/or  $M^n$ is a simply connected integer cohomology $\HP^{r-2}$ and $r=\frac{n}4+2$.

In the following theorem, we lower the rank bound from approximately $n/4$ to $n/6 +1$ and classify the fixed-point set components of the $\Z_2$-torus.
Since closed, positively curved manifolds of dimensions $2$ and $3$ are classified by the Gauss-Bonnet Theorem and work of Hamilton \cite{H} and Wolf \cite{Wo}, respectively, we only classify fixed-point set components of dimensions $4$ and above. 

\begin{TA}\label{TA} Let $M^n$ be a closed, positively curved manifold with $n\geq 7$, and assume $\Z_2^r$ acts effectively by isometries on $M$ with a fixed point, $x \in M$. Let $F^m_x$ be a connected component of the fixed-point set of $\Z_2^r$ containing $x$. Suppose that $m\geq 4$ and 
		\[r\geq\frac{n}{6}+1 \,.\]		
\noindent Then one of the following holds:
\begin{enumerate}
\item $F^m$ is homotopy equivalent to one of  $S^m$, $\RP^m$, $\CP^{\frac{m}2}$, or a lens space; or
\item  $\widetilde{F}^m$, the universal cover of $F^m$, has $1$-, $2$-, or $4$-periodic cohomology, that is, $\widetilde{F}^m$ has the integral cohomology ring of   $S^m$, $S^3 \times \HP^{\frac{m-3}{4}}$, $N_j^{m}$ with $m\equiv 2, 3 \pmod 4$, $\CP^{\frac{m}{2}}$, $S^2\times \HP^{\frac{m-2}{4}}$, $E^{m}_{\ell}$ with $\ell\geq 2$ and $m\equiv 2 \pmod 4$, or $\HP^{\frac{m}{4}}$.

\end{enumerate}
\end{TA}

\begin{remark}
   A list of all possible cohomology rings of  $F^m$ is given in Theorem \ref{Atech}. We refer the reader to Definition \ref{k-periodic} and Remark \ref{prdrmk} for the definition of an $n$-manifold with $k$-periodic cohomology, and to Definition \ref{en} for the definition of $E^{m}_{\ell}$ and $N_j^{m}$.
\end{remark}

 We can further lower the rank bound in the following theorem.

\begin{TB}\label{TB} 
Let $M^n$ be a closed, positively curved Riemannian manifold, and assume $\Z_2^r$ acts effectively by isometries on $M$ with a fixed point, $x\in M$, with $n\geq 9$. Suppose that 
		\[r\geq\frac{n}8+2.\]		
\noindent Let $F^m_x$ be a connected component of the fixed-point set of $\Z_2^r$ containing $x$.  Then $H^*(F^m_x;\Z_2)$ is isomorphic as a graded ring to $H^*(N^m/\Gamma;\Z_2)$, where $N^m$ is one of $S^m$, $\CP^{\frac{m}{2}}$, $\HP^{\frac{m}{4}}$, $S^2\times \HP^{\frac{m-2}{4}}$ or  $S^3\times \HP^{\frac{m-3}4}$ and $\Gamma$ is a subgroup of the generalized quaternion group $Q_{16}$ acting freely on $N^m$. 
\end{TB}

\begin{remark}
    We can actually show more in Theorem \hyperref[TB]{B} depending on the codimension of the inclusion of $F^m_x$ in a fixed point set component of a corank one subgroup of $\Z^r_2$ of minimal dimension. We list the possibilities that can occur  in Theorem \ref{Btech}.
\end{remark}

We briefly outline the strategy to prove Theorems \hyperref[TA]{A} and \hyperref[TB]{B}. We begin by establishing some notation. Let $F_j$ denote the fixed-point set component of a corank-$j$ subgroup of $\Z_2^r$ at $p$, with $0\leq j\leq r-1$. Let $m_j=\dim(F_j)$ and  $k_j$ denote $\codim(F_j\subset F_{j+1})$. Then, as observed in \cite{SW} (see Proposition \ref{chain}), one can construct an ascending chain of fixed-point set components $F_j$ such that for each $j$,  $k_j\leq k_{j+1}$. Since closed manifolds of positive curvature of dimension $\leq 3$ are classified and both results easily follow if the codimension is equal to one, one can then assume that $\dim(F_0)\geq 4$ and that $k_0\geq 2$.  Combining this information with the lower bounds on the rank of the $\Z_2$-torus in Theorems \hyperref[TA]{A} and \hyperref[TB]{B}, one sees that for small $j$ the codimensions of the corresponding $F_j \subset F_{j+1}$ are bounded between $2$ and $5$ and $2$ and $7$, respectively. Armed with these restrictions and generalizations of  the small codimension lemmas in \cite{KKSS} to almost all higher dimensions, one obtains the results.

\begin{org} This paper is organized as follows. We collect preliminary results in Section \ref{2}. In Section \ref{3}, we prove the higher codimension lemmas we need to decrease the bound on the rank of the $\Z_2$-torus in Theorems A and B. In Section \ref{4}, we prove various technical lemmas about the lower part of the chain of fixed-point set components. We then leverage these results in Sections \ref{5}  and \ref{6} to prove Theorems \hyperref[TA]{A} and \hyperref[TB]{B}, respectively.
    \end{org}
    
\begin{ack} A portion of this work draws from the PhD thesis of A. Bosgraaf. The authors are grateful to L. Kennard for numerous helpful conversations. C. Escher acknowledges support from the Simons Foundation (\#585481, C. Escher). C. Searle was partially supported by NSF Grant DMS-1906404 and DMS-2204324. This material is based upon work supported by the National Science
Foundation under Grant No. DMS-1928930, while C. Escher and C. Searle were in
residence at the Simons Laufer Mathematical Sciences Institute
(formerly MSRI) in Berkeley, California, during the Fall 2024
semester. 

\end{ack}

\section{Preliminaries}\label{2}

In this section we collect preliminary notions and results that we use in the proofs of Theorems \hyperref[TA]{A} and \hyperref[TB]{B}.

\subsection{Covering spaces}
We begin by 
 collecting some known results on covering spaces.  Throughout, we denote the universal cover of a manifold, $M$, by $\widetilde M$. We first state Proposition 4.1 from Hatcher \cite{Hat}.
 \begin{proposition}
\textup{\cite{Hat}}\label{covergroups}
The map $p_*:\pi_k(\widebar{X}, \widebar{x}_0)\to\pi_k(X, x_0)$ induced by a covering projection \newline $p :(\widebar{X}, \widebar{x}_0)\to(X, x_0)$ is injective if $k=1$, and an isomorphism if $k\geq 2$.
\end{proposition}
 To prove Theorems \hyperref[TA]{A} and \hyperref[TB]{B}, we often find ourselves in the situation where we need to lift a $k$-connected  embedded submanifold $N$ of $M$ to the universal cover of $M$.
Since 
the covering projection, $p:\widetilde{M}\rightarrow M$ is a local homeomorphism,  a connected component $\widebar{N}$ of $p^{-1}(N)$ is an embedded submanifold of $\widetilde{M}$. Consider the following commutative diagram:
\[\begin{tikzcd}
			{\pi_i(\widebar{N})} & {\pi_i(\widetilde{M})} \\
			\pi_i(N) & \pi_i(M)
			\arrow["p_*"', from=1-1, to=2-1]
			\arrow["p_*", from=1-2, to=2-2]
			\arrow["\iota_*"', from=2-1, to=2-2]
			\arrow["\tilde\iota_*", from=1-1, to=1-2]\,.
		\end{tikzcd}\]
Applying Proposition \ref{covergroups} and using the $k$-connectivity of $N \hookrightarrow M$, gives us that $\tilde\iota_*:\pi_i(\bar{N})\to \pi_i(\widetilde{M})$ is an isomorphism for $1\leq i\leq k-1$ and surjective for $i=k$. In particular,  we obtain the following result.

\begin{lemma}\label{connlifting} 
Let $M$ be a manifold,  $N\subseteq M$ an embedded submanifold of $M$ with $\dim(N)\geq 2$, and $p:\widebar{M}\to M$
 a covering projection.  Suppose $N \hookrightarrow M$ is $k$-connected with $k \ge 2$ and $\widebar{N}$ is a connected component of $p^{-1}(N)$.  Then $\widebar{N}\hookrightarrow \widebar{M}$ is $k$-connected, and if $\widebar{M}$ is the universal cover of $M$, then $p|_{\widebar N}:\widebar{N}\to N$ is the universal cover of $N$.
\end{lemma}

Proposition 3G.1 in \cite{Hat}, included below, uses the transfer homomorphism to give us information about the cohomology of a finite-sheeted covering space.
\begin{proposition}\textup{\cite{Hat}}\label{transfer} Let $\pi:\bar{X}\to X$ be an $n$-sheeted covering space. Then with coefficients in a field $\F$ whose characteristic is 0 or a prime, $p$, such that $p\nmid n$, the map 
$\pi^*:H^k(X;\F)\to H^k(\widebar{X};\F)$
is injective. 
\end{proposition} 
Theorem 3.4 from \cite{KKSS}, stated below, allows us to identify certain manifolds $M$ up to homotopy equivalence, provided we have information about their fundamental groups and the cohomology ring of their universal covers. 
\begin{theorem} \textup{\cite{KKSS}} \label{universal} 
Let $M^n$ be a closed smooth manifold. Then the following hold:
\begin{enumerate}
\item If $\pi_1(M)$ is cyclic and $\widetilde{M}$ is a $\Z$-cohomology sphere, then $M$ is homotopy equivalent to $S^n$, $\RP^n$, or a lens space $S^n/\Z_l$ for $l\geq 3$. 
\item If $M$ is simply connected and has the integer cohomology of $\CP^{n/2}$, then $M$ is homotopy equivalent to $\CP^{n/2}$.
\end{enumerate}
\end{theorem}

\subsection{Periodic cohomology}

We begin with the definition of a periodic cohomology ring as in the work of Nienhaus \cite{N}.  

\begin{definition}\cite{N}\label{k-periodic} Let $M$ be a connected topological space, $R$ a ring, and $c$ a positive integer. We say that $x\in H^k(M;R)$ induces $k$-periodicity up to degree $c$, or $M$ has $k$-periodic $R$-cohomology up to degree $c$, provided either
\lista
    \item $k\leq \frac{c}2$ \text{ and } $\smile x: H^i(M;R)\to H^{i+k}(M;R)$ is surjective for $0\leq i<c-k$ and injective for $0< i\leq c-k$; or,
    \item $\frac{c}{2}<k\leq c$ and $x$ is a product of periodicity-inducing elements of degree  $\leq \frac{c}{2}$. 
\end{enumerate}
If $M$ is a closed manifold of dimension $c$, then we say that $M$ has $k$-periodic $R$-cohomology.
    
\end{definition}

\begin{remark}\label{prdrmk}
The definition of a $k$-periodic cohomology ring of a manifold used in \cite{Kennard} stated that a manifold was $k$-periodic provided Part 1 of Definition \ref{k-periodic} held for any $k\leq n$. There are problems with this definition, as pointed out in \cite{N}, if $k>n/2$. In this paper, we are only concerned with manifolds that are at most four-periodic, so by convention, $6$- and $7$-dimensional 
manifolds satisfying the previous definition for $k=4$ are also called four-periodic.
\end{remark}

The \hyperref[periodicity]{Periodicity Lemma}  stated below, for cohomology rings with either integral coefficients due to \cite{Wi} or $\Z_2$ coefficients due to \cite{N},  gives information about  the cohomology ring of a closed manifold $M$ with a highly connected submanifold $N$. 

\begin{period}\textup{ \cite{Wi, N}} \label{periodicity}
Let $R$ be $\Z$ or $\Z_2$ and $M^{n+k}$ a closed differentiable $R$-oriented manifold. Suppose $N^{n}$ is an embedded, closed, $R$-oriented submanifold of $M$ such that the inclusion $\iota :N^{n}\to M^{n+k}$ is $(n-l)$-connected and $n-2l>0$. Let $\iota_*[N]\in H_{n}(M;R)$ be the image of the fundamental class of $N$ in $H_{*}(M;R)$, and let $x\in H^k(M;R)$ be its Poincar\'e dual. Then the homomorphism
\[\smile x:H^{i}(M;R)\to H^{i+k}(M;R)\]
is surjective for $l\leq i<n-l$ and injective for $l<i\leq n-l$. In particular, $\smile x$ is an isomorphism for $l<i<n-l$.

\end{period}

In the Periodicity Lemma, note that when $l=0$, that is, when $N$ is maximally connected in $M$, and both are $R$-oriented, then $M$ has $k$-periodic cohomology with coefficients in $R$. 

\begin{remark}\label{surj}  If $M$ is connected and $H^*(M;R)$ is $k$-periodic, then
		\[\mathrm{rank}_R\(H^{ik}(M;R)\)\leq \mathrm{rank}_R\(H^0(M;R)\) = 1\,,\]		
\noindent and, furthermore, if $R$ is a field or $\Z$, then $x^i$ generates $H^{ik}(M;R)$ for $0<i<c/k$, where $c$ is the largest degree of $H^*(M; R)$ for which some non-trivial element induces periodicity.
\end{remark}

When $R=\Z_p$ or more generally a field, periodicity can be refined in some cases.
For cohomology with coefficients in an arbitrary ring, Lemma 2.5 from \cite{N} characterizes elements cupping to a periodicity-inducing element.

\begin{lemma}{\cite{N}} \label{factor} Let $R$ be a ring with identity.
 If $x\in H^*(M; R)$ induces periodicity and $x = y\smile z$, then $y$ and $z$ also induce periodicity.
 \end{lemma}

Combining Proposition 1.3 and Lemma 3.2 from \cite{Kennard} yields the following $\Z_2$-periodicity result.

\begin{theorem}[{$\Z_2$-periodicity Theorem}] \textup{\cite{Kennard}} \label{z2}
Suppose $x\in H^k(M;\Z_2)$ and $y\in H^l(M;\Z_2)$ are non-trivial and induce periodicity in $H^*(M;\Z_2)$ up to degree $c$ with $c\geq 2k$. If $y$ has minimal degree among all such elements, then $l$ is a power of $2$ and $l$ divides $k$. 
\end{theorem}

Throughout this article we say a manifold $M$ {\it is an $R$-cohomology $X$} for a given topological space $X$ or {\it has the $R$-cohomology ring of $X$}, if $M$ and $X$ have isomorphic  cohomology rings, that is  $H^*(M;R) \cong H^*(X;R)$.

In \cite{KKSS},   they give a list of the possible cohomology rings that correspond to closed, simply connected, four-periodic, Poincar\'e duality topological spaces. 
Besides the CROSSes and $S^3\times \HP^m$, the only other two possibilities are defined as follows. 
\begin{definition}\cite{KKSS}\label{en} Define  $E_\ell^{4m+2}$ and $N_j^{n}$, $\ell \ge 0, j \ge 2, j, \ell \in \Z$, and $n\equiv 2, 3 \mod 4$, to be closed, simply connected manifolds whose four-periodic integral cohomology ring is given  by 

$$\begin{aligned}
    H^*(E_\ell^{4m+2}; \Z) &\cong \Z[x, y]/((y^2-{\ell}x), x^{m+1},y^{2(m+1)}),\,\deg(y)=2, \deg(x)=4.\\
H^i(N_j^{2+4m};\Z)&\cong \left\{\begin{array}{ll} 0&\text{for } i\equiv 1,2\pmod{4}\\ 
    \Z_j&\text{for } i\equiv 0, 3\pmod 4\end{array}\right., \\
H^i(N_j^{3+4m};\Z)&\cong \left\{\begin{array}{ll} 0&\text{for } i\equiv 1,2,3\pmod{4}\\ 
    \Z_j&\text{for } i\equiv 0\pmod 4\end{array}\right.
\end{aligned}$$
 \end{definition}
\begin{remark}   $E_0$ is an integral cohomology $S^2 \times \HP^m$ and  $E_1$ is an integral cohomology $\CP^{2m+1}$. It is however, unknown whether $E_{\ell}$ exists for $\ell\geq 2$, nor is it known whether the $N^{3+4m}_j$ exist. We show in Lemma \ref{4z}, that $N^{2+4m}_j$ cannot occur.
\end{remark}

For closed, simply connected, four-periodic $n$-manifolds the isomorphism classes of their $\Z_2$-cohomology rings  can  be described completely.

\begin{proposition} \label{4z2} 
If $M^n$ is a simply connected, closed $n$-manifold such that $H^*(M;\Z_2)$ is four-periodic with $n\geq 8$, then $M$ has the $\Z_2$-cohomology ring of $S^n$, $\CP^{n/2}$, $\HP^{n/4}$, $S^2\times\HP^{\frac{n-2}4}$, or $S^3\times\HP^{\frac{n-2}4}$.
\end{proposition}

\begin{proof} If $H^*(M;\Z_2)$ has periodicity of degree $k<4$, then $k\mid 4$ by Theorem \ref{z2}, and so $H^*(M;\Z_2)$ is $1$- or $2$- periodic. If $H^*(M;\Z_2)$ is $1$-periodic then $M$ has the $\Z_2$-cohomology ring of $S^n$. If $H^*(M;\Z_2)$ is $2$-periodic and the periodicity-inducing element is non-trivial, then $M$ has the $\Z_2$-cohomology ring of
$\CP^{n/2}$.  

Suppose that $4$ is the minimal degree of periodicity in $H^*(M;\Z_2)$. If the periodicity-inducing element $x\in H^4(M;\Z_2)$ is trivial, then $M$ has the $\Z_2$-cohomology of $S^n$, so we assume $x$ is non-trivial.
    If $n\equiv 0,1\pmod 4$, 
    it follows from 4-periodicity, simple-connectivity, Poincar\'e duality, and Lemma \ref{factor} 
    that $M$ has the $\Z_2$-cohomology ring of $S^n$, $\CP^{n/2}$, or $\HP^{n/4}$. 
    We break the remainder of the proof into two cases: Case 1, where  $n\equiv 3\pmod 4$ and Case 2, where $n\equiv 2\pmod 4$.

     \noindent{\bf Case 1, where ${\bf n\equiv 3\pmod 4}$:} We claim that $H^*(M;\Z_2)\cong H^*(S^3\times\HP^{\frac{n-3}4};\Z_2)$. To prove the claim, we note that simple-connectivity, Poincar\'e duality, and four-periodicity give us that 
    
    \[H^i(M;\Z_2)\cong \begin{cases} 0 \hspace{5mm}\text{ for } i\equiv 1,2\pmod{4} \\
    \Z_2 \hspace{2.8mm}\text{ for } i\equiv 0,3\pmod 4 \\
    \end{cases}.\]
Since $H^6(M;\Z_2)\cong 0$, if $y$ generates $H^3(M;\Z_2)$ then $y^2=0$. If we denote by $x$ the generator of $H^4(M;\Z_2)$, we have by four-periodicity that
        \[H^*(M;\Z_2)\cong \Z_2[x,y]/(y^2,x^{\frac{n+1}4})\cong H^*(S^3\times\HP^{\frac{n-3}4};\Z_2).\]

\noindent{\bf Case 2, where ${\bf n\equiv 2\pmod 4}$:} We claim that $H^*(M;\Z_2)$ is isomorphic to $ H^*(S^2\times\HP^{\frac{n-3}4};\Z_2)$. By simple-connectivity, Poincar\'e duality, and four-periodicity, we have that 
\begin{equation}\label{e1}
H^i(M;\Z_2)\cong \begin{cases} 0 \hspace{5mm}\text{ for } i\equiv 1\pmod{4} \\
    \Z_2 \hspace{2.8mm}\text{ for } i\equiv 0,2\pmod 4 \\
    \Z^l_2 \hspace{2.8mm}\text{ for } i\equiv 3\pmod 4
    \end{cases}
    \end{equation}
    for some $l\geq 0$. To prove our claim, we first show that $l=0$. To do this, we expand on ideas from Section 6 of \cite{Kennard} utilizing the      Steenrod algebra and argue by contradiction. 
    Suppose then that $l>0$ and let $w\in H^2(M;\Z_2)$ be a generator. Since $H^3(M;\Z_2)$ is non-trivial, we can choose a non-trivial $z\in H^3(M;\Z_2)$. Using Poincar\'e duality and four-periodicity, there is a $y\in H^3(M;\Z_2)$, so that $$y\smile z=x\smile w.$$ We can then use the fact that $x \in H^4(M;\Z_2)$ is the minimal periodicity-inducing element, together with the Cartan formula, the Steenrod axioms, the Adem relation $\mathrm{Sq}^3=\mathrm{Sq}^1 \mathrm{Sq}^2$, and the Bockstein homomorphism associated to $0 \to \Z_2\to\Z_4\to\Z_2 \to 0$, to prove that $\mathrm{Sq}^4(y\smile z)=0$ and $\mathrm{Sq}^4(x\smile w)=x^2\smile w\neq 0$, a contradiction. Thus, 
 \begin{equation}\label{e}H^i(M;\Z_2)\cong \begin{cases} 0 \hspace{5mm}\text{ for } i\equiv 1,3\pmod{4} \\
    \Z_2 \hspace{2.8mm}\text{ for } i\equiv 0,2\pmod 4 \\
    \end{cases}.
    \end{equation}
There are now just two possibilities for $w^2$: either $w^2=x$ or $w^2=0$. By Lemma \ref{factor} and the assumption that $x \in H^4(M;\Z_2)$ is the periodicity-inducing element of minimal degree, the first case does not occur.
In the second case we obtain $H^*(M;\Z_2)\cong H^*(S^2\times\HP^{\frac{n-3}4};\Z_2)$. 
\end{proof}

With this result in hand, we can now classify those integral cohomology rings of closed, simply connected, $n$-manifolds that are four-periodic when $n\geq 8$ in the next lemma.

 \begin{lemma} \label{4z} 
If $M^n$ is a simply connected, closed $n$-manifold such that $H^*(M;\Z)$ is four-periodic with $n\geq 8$, then $M$ has the integral cohomology ring of $S^n$, $\CP^{n/2}$, $\HP^{n/4}$, $S^3\times\HP^{\frac{n-2}4}$, $N_j^{n}$, $n\equiv 3 \pmod 4$ or $E_\ell^{n}$, $\ell\geq 2$, $N_j^n$, $n\equiv 2 \pmod 4$,  and $j$ odd.
\end{lemma}

\begin{proof} If $H^*(M;\Z)$ has periodicity of degree $k<4$, then 
$H^*(M;\Z)$ is $1$- or $2$-periodic  by Lemma \ref{factor}. Since $\pi_1(M)$ is trivial, if $H^*(M;\Z)$ is $1$-periodic, then $M$ has the $\Z$-cohomology ring of $S^n$. If $H^*(M;\Z)$ is $2$-periodic and the periodicity-inducing element is non-trivial, then $M$ has the $\Z$-cohomology ring 
$\CP^{n/2}$.  

Suppose now that $4$ is the minimal degree of periodicity in $H^*(M;\Z)$. If the periodicity-inducing element $x\in H^4(M;\Z)$ is trivial, then $M$ has the $\Z$-cohomology of $S^n$, so we assume from now on that $x$ is non-trivial.
We break the rest of the proof into cases according to the parity of the dimension of $M$ $\pmod 4$.

\noindent {\bf Case 1, where ${\bf n\equiv 0,1\pmod 4}$}: 
Since $\pi_1(M)=0$, it follows from 4-periodicity,  Poincar\'e duality, and Lemma \ref{factor},
that $M$ has the $\Z$-cohomology ring of $S^n$,  $\CP^{n/2}$, or $\HP^{n/4}$.

 \noindent{\bf Case 2, where ${\bf n\equiv 3\pmod 4}$:} We claim that $H^*(M;\Z)\cong H^*(S^3\times\HP^{\frac{n-3}4};\Z)$ or $H^*(N^n_j; \Z)$. To prove the claim, we note that simple-connectivity, Poincar\'e duality, and four-periodicity give us that for $0<i<n$, either 
 $$H^i(M;\Z)\cong \left\{\begin{array}{ll} 0&\text{ for } i\equiv 1,2\pmod{4}\\ 
    \Z &\text{ for } i\equiv 0,3\pmod 4\end{array}\right.$$
   or
$$\hspace{.5cm} H^i(M;\Z)\cong \left\{\begin{array}{ll} 0&\text{for } i\equiv 1,2,3\pmod{4}\\ 
    \Z_j&\text{for } i\equiv 0\pmod 4\end{array}\right.,$$
depending on whether $H^4(M;\Z)\cong \Z$ or $\Z_j$, respectively.

If $H^i(M; \Z)\cong \Z$ for $i\equiv 0, 3\pmod 4$, we argue as follows. Since $H^6(M;\Z)\cong 0$, if $y$ generates $H^3(M;\Z)$ then $y^2=0$. If we denote by $x$ the generator of $H^4(M;\Z)$, we have by four-periodicity that
$$H^*(M;\Z)\cong \Z[x,y]/(y^2,x^{\frac{n+1}4})\cong H^*(S^3\times\HP^{\frac{n-3}4};\Z).$$

If $H^i(M; \Z)\cong \Z_j$ for $i\equiv 0 \pmod 4$, we denote by $x$ the generator of $H^4(M;\Z)$ and see that we have 
\[H^*(M;\Z)\cong H^*(N_j^n;\Z).\]

\noindent{\bf Case 3, where ${\bf n\equiv 2\pmod 4}$:} We claim that $H^*(M;\Z)$ is isomorphic to the integral cohomology ring of $N^n_j$ or $E_{\ell}^n$, recalling that $H^*(E_0; \Z)\cong H^*(\CP^{n/2}; \Z)$ and $H^*(E_1; \Z)\cong H^*(S^2\times \HP^{\frac{n-2}{4}})$. Since $\pi_1(M)=0$, Poincar\'e duality and  four-periodicity give us 

\begin{equation}\label{e2}
H^i(M;\Z)\cong \begin{cases} 0 \hspace{5mm}\text{ for } i\equiv 1, 3\pmod{4} \\
   \Z  \hspace{4.2mm}\text{ for } i\equiv 0, 2\pmod 4 \\
    \end{cases}, \text{ or }
    \end{equation}

 \begin{equation}\label{e4}
H^i(M;\Z)\cong \begin{cases} 0 \hspace{5mm}\text{ for } i\equiv 1, 2\pmod{4} \\
   \Z_j  \hspace{2.8mm}\text{ for } i\equiv 0, 3\pmod 4 \\
    \end{cases}.
    \end{equation}
\noindent In Display \ref{e2} we have $H^*(M;\Z)\cong H^*(E^n_l;\Z)$ and in Display \ref{e4} we have $H^*(M;\Z)\cong H^*(N^n_j;\Z_2)$. In the latter case, we claim that $j$ must be odd. If instead $j$ is even,  $H^3(M;\Z_2)\cong \Z_2^2$ by Poincar\'e duality and the Universal Coefficient Theorem. Using the the universal coefficient theorem, Poincar\'e duality and the Bockstein homomorphism one can show that for $2k \le n$ $k$-periodicity of $H^*(M;Z)$ implies $k$-periodicity of $H^*(M;Z_p)$ for any prime $p$. Thus   $H^*(M;\Z_2)$ is 4-periodic, and so $H^*(M;\Z_2)$ is as in the conclusion of Proposition \ref{4z2}. This implies that $H^3(M;\Z_2)\cong 0$, a contradiction. Hence $j$ must be odd, and the proof is complete.
\end{proof}

\subsection{Group Cohomology} 
In this section we recall some basic results on group cohomology. For more details on this subject, we refer the reader to \cite{Bn}.  

The cohomology of a group $G$ is isomorphic to the cohomology of the classifying space $BG\cong K(G,1)$ of $G$, where the coefficients can be taken in any abelian group.    
For $G=\Z_k$, $B\Z_k$ is the infinite dimensional lens space $L^{\infty}_k$.  The cohomology ring of $L^{\infty}_k$
with $\Z_k$-coefficients can be computed from the cup product structure of $\CP^{\infty}$, see for example \cite{Hat}.  The ring structure depends on the 
parity of $k$ and can be summarized as follows.  
	$$H^*(\Z_k;\Z_k)\cong H^*(L_k^{\infty};\Z_k)\cong\begin{cases}
		   \Z_k[x,y]/(x^2),\, \,\deg(x)=1,\deg(y)=2,y = \beta(x)& k \,\text{odd},\\
 \Z_k[x,y]/(2x^2), \,\deg(x)=1,\deg(y)=2, y = \beta(x) & k \,\text{even},
                   \end{cases}$$
\noindent where $\beta$ is the Bockstein homomorphism.

Using the Universal Coefficient Theorem for $k$ odd and the ring homomorphism, $H^*(L_k^{\infty};\Z_k)\to H^*(L_k^{\infty};\Z_2)$, induced by the projection $\Z_k\to\Z_2$ for $k$ even, we may calculate the cohomology ring of $\Z_k$ with $\Z_2$ coefficients.  The cohomology ring of the quaternion group and the generalized quaternion group have been calculated in \cite{AM}, for example.  We summarize these two  results in the following lemma.  

\begin{lemma} \label{cohomQ} Let $G$ be either $\Z_k$ or the generalized quaternion group, $Q_{2^i}$  of order $2^i$. Then the group cohomology rings are given as follows:
 
\begin{enumerate} 
\item \cite{Hat}  If $G=\Z_k$, we have
$$H^*(\Z_k; \Z_2)\cong\begin{cases}

\Z_2,\,\, & k \,\,\text{odd},\\
   
 \Z_2[x], \,\,\deg(x)=1&k\equiv 2\pmod 4, \\
 \Z_2[x,y]/\(x^2\),\,\,\deg(x)=1,\,\,\deg(y)=2& k\equiv 0\pmod 4.
                   \end{cases}
$$ 
  \item\cite{AM} If $G=Q_{2^i}$, we have 
  $H^*(Q_{2^i};\Z_2)$
  is periodic with period $4$, and 
$$H^*(Q_{2^i};\Z_2)\cong\begin{cases} 
 \Z_2[x,y,e]/\(x^3, y^3,\,x^2y+xy^2\), \,\,\deg(x)=\deg(y)=1,\deg(e) = 4 & i=3;\\
 \Z_2[x,y,e]/\(xy,\,x^3+y^3\), \,\, \deg(x)= \deg(y)=1,\deg(e) = 4& i\geq 4.
       \end{cases} $$
\end{enumerate}
\end{lemma}

\noindent In particular, we see that the $\Z_2$-cohomology groups of $Q_{2^i}$, $i\geq 3$, satisfy the following:

        \[H^l(Q_{2^i};\Z_2)\cong\begin{cases} \Z_2\hspace{4mm} \text{ if } l\equiv 0,3\pmod 4\\ \Z_2^2\hspace{4mm} \text{ if } l\equiv 1,2\pmod 4.\end{cases}\]

We recall the $\Z_2$-cohomology ring of the finite dimensional lens spaces in the following lemma.

\begin{lemma}\cite{Hat} \label{cohomlens} Let $L_k^m$ denote a lens space with fundamental group $\Z_k$, $k\geq 2$, and dimension $m$. Then
    $$H^*(L_k^m; \Z_2)\cong\begin{cases}

\Z_2[x]/(x^2),\,\,\deg(x)=m & k \,\,\text{odd},\\
   
 \Z_2[x]/(x^{m+1}), \,\,\deg(x)=1&k\equiv 2\pmod 4, \\
 \Z_2[x,y]/(x^2, y^{\frac{m+1}2}),\,\,\deg(x)=1,\,\,\deg(y)=2& k\equiv 0\pmod 4.
                   \end{cases}
		$$ 
\end{lemma}
By Lemma \ref{cohomlens}, the $\Z_2$-cohomology ring of the $L_k^m$ fall into one of three isomorphism types:  that of $S^m$, $\RP^m$, or $S^1\times \CP^{\frac{m-1}{2}}$,  
respectively. We now recall some basic facts about equivariant cohomology, the setting in which we apply group cohomology. 

Let $G$ be a compact Lie group and let $X$ be a $G$-space. The equivariant cohomology of $X$, denoted $H_G^*(X;R)$ is defined as
		\[H^*_G(X;R):=H^*(EG\times_GX;R),\]
where $EG\times_GX$ is the Borel construction, and $EG \longrightarrow BG$ is the universal principal $G$-bundle.
If $G$ acts freely on $X$, then the canonical map $EG\times_GX\to X/G$ is a homotopy equivalence, and so $H^*_G(X;R) \cong H^*(X/G;R)$. We can then consider the Borel fibration, $X\to EG\times_GX\to BG$, and its associated Serre spectral sequence. Let $\F$ be a field and assume that the action of $\pi_1(BG)$ on 
$H^*(X; \F)$ is trivial.  Then 
\[H^p\(BG;H^q(X;\F)\)\cong H^p(BG;\F)\otimes_{\F} H^q(X;\F)\]
by the Universal Coefficient Theorem.  
In this case the Serre spectral sequence has 
$E_2$-page 
		\[E_2^{p,q}\cong H^p(BG;\F)\otimes H^q(X;\F),\]
and it converges to $H^*(EG\times_GX;\F) \cong H^*(X/G;\F)$. This is our main  
 tool for computing the $\Z_2$-cohomology ring of quotients.
\subsection{Finite group actions}

Our goal in what follows is to understand which fundamental groups may appear for closed manifolds whose universal cover has $k$-periodic cohomology, for $k\in \{1, 2, 4\}$. This involves understanding which finite groups may act freely.

Recall that for closed, positively curved manifolds, in even dimensions the fundamental group is either $\Z_2$ or trivial by Synge's theorem and in odd dimensions it is finite by Bonnet-Myer's theorem.
For topological reasons,  when $M$ is an integral cohomology $\CP^{m/2}$, with $m\equiv 0 \pmod 4$, 
it is known that there is no free action by any non-trivial
group, see Example 2.6 in Cusick \cite{Cu}. Combining these facts  with Lemma \ref{4z}, we obtain a complete list of 
 closed manifolds whose universal cover has $k$-periodic integral cohomology ring for $k\in\{1, 2, 4\}$. We summarize this information  in the following proposition.

\begin{proposition} \label{quotients} Let $\widetilde{M}^n$ be a closed, positively curved manifold with $k$-periodic integral cohomology, $k\in \{1, 2, 4\}$. Then 
$M^n$ has the integral cohomology ring of   $S^n/\Gamma$, $(S^3/\Gamma) \times \HP^{\frac{n-3}{4}}$, $N_j^{n}/\Gamma$,  $n\equiv 3 \pmod 4$, $\CP^{\frac{m}{2}}/\Delta$, $(S^2/\Delta)\times \HP^{\frac{n-2}{4}}$, $E^n_{\ell}/\Delta$, $n\equiv 2 \pmod 4$ and $\ell\geq 2$, or $\HP^{\frac{n}{4}}$, where $\Gamma$ is a finite group acting freely on the corresponding manifold  and $\Delta$ is  a subgroup of $\Z_2$.
\end{proposition}
We now wish to identify closed, simply connected spaces with $k$-periodic cohomology, $k\in \{1, 2, 4\}$, that do not admit a free $(\Z_2\times \Z_2)$-action. The first class of such spaces are $\Z_p$-cohomology spheres, as shown in Theorem III.8.1 in \cite{Br}. 
\begin{proposition}\cite{Br}\label{zpzp} Let $p$ be a prime.  Then there is no free action of  $\Z_p\times \Z_p$ on a finitistic $\Z_p$-cohomology $n$-sphere for any $n\geq 1$.
\end{proposition}

We also prove the following lemma about free $(\Z_2\times \Z_2)$-actions
on a $\Z_2$-cohomology $S^3\times \HP^m$.

\begin{lemma}\label{freez2}
 Let $M$ be a closed manifold 
 such that $H^*(\widetilde{M};\Z_2)\cong H^*(S^3\times \HP^m;\Z_2)$ with $m\geq 2$. If $\Z_2\times \Z_2$ acts freely on $\widetilde{M}$, then $\widetilde{M}/(\Z_2\times \Z_2)$ does not have a $1$-, $2$-, or $4$-periodic $\Z_2$-cohomology ring.
\end{lemma}

\begin{proof} Let $\Z_2\times \Z_2$ act freely on $\widetilde{M}$ and let $\widebar{M}=\widetilde{M}/(\Z_2\times\Z_2)$.  Assume that $H^*(\widebar{M};\Z_2)$ is $k$-periodic, $k\in\{1, 2, 4\}$, to obtain a contradiction.  By Lemma \ref{cohomQ} and the K{\"u}nneth Formula we know that
\[H^{j}(\Z_2^2;\Z_2)\cong H^{j}(\RP^{\infty}\times\RP^{\infty};\Z_2)\cong \Z_2^{j+1}\,\textrm{for all} \,j.\] 		
\noindent Consider the Serre spectral sequence associated to the Borel fibration, $\widetilde{M}\to E{\Z_2^2}\times_{\Z_2^2} \widetilde{M}\to B\Z_2^2$. Since $\rk(H^i(\widetilde{M};\Z_2))\leq 1$ for all $i$, we have that $\pi_1(BG)$ acts trivially on $H^*(\widetilde{M};\Z_2)$. Therefore $E_2^{j,0}\cong \Z_2^{j+1}$ for all $j \geq 0$. Note that the $E_l^{r, s}$ are trivial for $r<0$. Since $\widetilde{M}$ is simply connected the $E_l^{0, 1}$ are also trivial. Hence the differentials,
$$d_l:E_l^{1-l,\,l-1}\to E_l^{1,0} \text{ and } d_l:E_l^{2-l,\,l-1}\to E_l^{2,0},$$ 
are both trivial for $l\geq 2$. Therefore $E_2^{1,0}=E_{\infty}^{1,0}$ and $E_2^{2,0}=E_{\infty}^{2,0}$, and hence 
		\[H^1\(\widebar{M};\Z_2\)\cong E_{2}^{1,0}\cong \Z_2^2\text{ and }H^2\(\widebar{M};\Z_2\)\cong E_{2}^{2,0}\cong  \Z_2^3.\] 	
\noindent However, since $H^*(\widebar{M};\Z_2)$ is $k$-periodic with $k\in\{1,2,4\}$ by hypothesis and   $\dim(M)\equiv 3\pmod 4$, by Poincar\'e duality we have 
		\[H^1\(\widebar{M};\Z_2\) \cong  H^2\(\widebar{M};\Z_2\),\]
\noindent a contradiction. 
\end{proof}

 We now recall the following basic fact from group theory, see, for example, Proposition 9.7.3 in Scott \cite{Sc}.

\begin{theorem}\cite{Sc}\label{pgroup}
If $G$ is a finite $p$-group with a unique subgroup of order $p$, 
then $G$ is cyclic or generalized quaternion. 
\end{theorem}
\noindent Theorem XII.11.6 in Cartan and Eilenberg \cite{CE} gives us the following characterization of a finite group with periodic cohomology.
\begin{theorem}\textup{\cite{CE}}\label{syl}  A finite group, $G$, has periodic cohomology with strictly positive period if and only if the Sylow $p$-subgroups, $Syl_p(G)$, of $G$ are cyclic or perhaps (if $p = 2$) generalized quaternion.
\end{theorem}

By Theorem \ref{pgroup}, for any finite group $G$ that contains no $\Z_p\times \Z_p$ subgroup,  it follows that $Syl_p(G)$  is either cyclic or, when $p=2$, generalized quaternion. If this is true for all $p$, Theorem \ref{syl} then tells us that $G$ has periodic cohomology and we obtain the following corollary for the fundamental group of $M$.
\begin{corollary}\label{sylc} Let $M$ be a manifold with finite fundamental group. Suppose $\widetilde{M}$  does not admit a free action of $\Z_p\times \Z_p$. Then  $Syl_p(\pi_1(M))$ is either cyclic or generalized quaternion. Moreover, if we assume this to be true for all primes $p$, then $\pi_1(M)$ has periodic cohomology.
\end{corollary}

We now recall Propositions 2.3 and 2.5 in Su \cite{Su}, which characterize the $\Z_2$-cohomology ring of the quotient of a free $\Z_2$-action on a $\Z_2$-cohomology sphere.
\begin{proposition}\textup{\cite{Su}}\label{Su}
Let $\Z_2$ act freely on $X$. Then $X$ is a  $\Z_2$-cohomology $n$-sphere if and only if $X/\Z_2$ is a $\Z_2$-cohomology real projective $n$-space.
\end{proposition}

 In the following proposition we utilize the Serre spectral sequence, first recalling that if $G$, a finite group,  acts freely on $S^n$, then by Proposition \ref{zpzp} and Corollary \ref{sylc}, $Syl_2(G)$ is either cyclic or generalized quaternion.  We note that the results in Proposition \ref{s3quotient} below were already known to be true, see Tomoda and Svengrowski \cite{TZ}. We include a proof that uses different techniques and significantly streamlines the previous proof. We observe that it also allows one to immediately read off the cohomology groups of the manifolds in question, as one sees in Corollary \ref{cors3}.

\begin{proposition}\label{s3quotient}
Let $G$ be a finite group acting freely on $S^3$. Then one of the following must hold.
\begin{enumerate}
		\item If $Syl_2(G)$ is cyclic, then $H^*(S^3/G;\Z_2) \cong H^*(S^3/\Z_k;\Z_2)$ with $k\in\{1,2,4\}$;
		\item If $Syl_2(G)$ is generalized quaternion, then $H^*(S^3/G;\Z_2) \cong H^*(S^3/\Gamma;\Z_2)$ where  $\Gamma$ is a subgroup of $Q_{16}\,.$ 
\end{enumerate}
\end{proposition}
\begin{proof}
 We first prove Part 1, when 
$Syl_2(G) \cong \Z_{2^i}$.  Consider the orbit space of the action of $Syl_2(G)$ on $S^3$.  
Since the action is free, $S^3/Syl_2(G)$ is a manifold, hence homeomorphic to a lens space. 
It follows that $S^3/Syl_2(G)\to S^n/G$ is an odd sheeted covering, and it follows from Proposition \ref{transfer} that $H^i(S^3/G;\Z_2)$ injects into $H^i(S^3/Syl_2(G);\Z_2)$ for all $i$.  Hence $S^3/G$ is a $\Z_2$-cohomology sphere or lens space, that is, $H^*(S^3/G;\Z_2)\cong H^*(S^3/\Z_k;\Z_2)$ with $k\in\{1,2,4\}$.

We now prove Part 2. Suppose $Syl_2(G)\cong 
Q_{i}$.  Consider the free action of $Q_i$ on $S^3$ and the Serre spectral sequence associated to the Borel fibration, $S^3\to E{Q_i}\times_{Q_i} S^3\to BQ_i$, and recall the cohomology ring of $BQ_i$ from Lemma \ref{cohomQ}.
 Since all $\Z_2$-cohomology groups of $S^3$ are either trivial or $\Z_2$, we observe that $\pi_1(BQ_i)$ acts trivially on $H^*(S^3;\Z_2)$. We then see that $E_1=E_2=E_3$ and $E_4=E_{\infty}$, so we need only compute $d_4$.  
  Since $H^i(S^3/G;\Z_2)=0\,\,\textrm{for all}\,i\ge 4$, it follows that $d_4:E^{j,3}_4\to E_4^{4+j,0}$ must be an isomorphism for all $j\geq 0$.
  By Lemma \ref{cohomQ}, we then have 
		\[H^k(S^3/Q_i;\Z_2)\cong\begin{cases}\Z_2\text{ for } k=0,3 \\ \Z^2_2\text{ for } k=1,2\end{cases},\]
where the ring structure on $H^*(S^3/Q_i;\Z_2)$ is determined by the cup products of elements in the first three cohomology groups of $H^*(Q_i;\Z_2)$.
It follows also from Lemma \ref{cohomQ} that there are only two possible isomorphism types, that of $H^*(S^3/Q_8;\Z_2)$ and that of $H^*(S^3/Q_{16};\Z_2)$. By Proposition \ref{transfer} and Poincar\'e duality, it follows that 
		\[\rk\(H^1(S^3/G;\Z_2)\)=\rk\(H^2(S^3/G;\Z_2)\)\in\{0,1,2\}.\] 
 \noindent We conclude that $H^*(S^3/G;\Z_2)\cong H^*(S^3;\Z_2)$, $H^*(S^3/\Z_k;\Z_2)$ for $k\in\{2,4\}$, or $H^*(S^3/Q_i;\Z_2)$ with $i\in\{8,16\}$.
\end{proof}

Since every closed $3$-manifold with positive sectional curvature is homotopy equivalent to a spherical space form $S^3/\Gamma$, we can use Proposition \ref{s3quotient} to determine the $\Z_2$-cohomology ring of such manifolds.
\begin{corollary}\label{cors3} Any positively curved closed 3-manifold $M$ has the $\Z_2$-cohomology ring of an $S^3/\Gamma$ where $\Gamma\leq Q_{16}$, the generalized quaternion group of order $16$. In particular, $M$ is a $\Z_2$-cohomology $S^3$, $\RP^3$, $S^3/\Z_4$, or $S^3/Q_i$ with $i\in\{8,16\}$.
\end{corollary}

Finally, we classify the four-periodic $\Z_2$-cohomology rings of closed manifolds all of whose covers have $k$-periodic $\Z_2$-cohomology, $k\in\{1, 2, 4\}$ in the following theorem.

\begin{remark} Throughout the remainder of this article, we write $M\sim_{\Z_2}N$ to denote that $M$ and $N$ have isomorphic $\Z_2$-cohomology rings, that is,  $H^*(M;\Z_2)\cong H^*(N;\Z_2)$.
\end{remark}
\begin{theorem}\label{z2anypi1}\text
    Let $M^n$ be a closed, positively curved $n$-manifold with $k$-periodic $\Z_2$-cohomology. 
    \lista
    
    \item If $k\in\{1,2\}$ and $n\geq 4$, then $M$ has the $\Z_2$-cohomology ring of $S^{n}$, $\CP^{n/2}$, $\RP^n$, or $L_4^n$.
    
   \item  If $k=4$, $n\geq 8$, and for any cover $\widebar{M}\to M$, we have that $\widebar{M}$ also has four-periodic $\Z_2$-cohomology then $M$ has the $\Z_2$-cohomology ring of $S^{n}$, $\CP^{n/2}$, $\HP^{n/4}$,  $S^2\times \HP^{\frac{n-2}4}$, $S^3\times\HP^{\frac{n-3}4}$, $L_k^n$ for $k\in\{2,4\}$, $\RP^2\times \HP^{\frac{n-2}4}$, $L_k^3\times\HP^{\frac{n-3}4}$ with $k\in\{2,4\}$, or $S^{n}/Q_i$ for $i\in\{8,16\}$ and $n\equiv 3\pmod{4}$.
    \listb
\end{theorem}

\begin{proof}
 
Throughout the proof we are working with $\Z_2$-coefficients and as closed manifolds are  $\Z_2$-orientable, we may use tools 
such as Poincar\'e duality. 
We break the proof into two parts, Part 1 where $ k\in\{1,2\}$ and $n\geq 4$ and Part 2 where $k=4$ and $n\geq 8$.

We begin with the proof of Part 1, where we assume that 
$k\in\{1,2\}$ and $n\geq 4$.
 If the periodicity-inducing element in $H^*(M;\Z_2)$ is trivial, it follows from the $1$- or $2$-periodicity that $M$ is a $\Z_2$-cohomology sphere. Suppose that the periodicity-inducing element in $H^*(M;\Z_2)$ is non-trivial. 
If $H^*(M;\Z_2)$ is $1$-periodic, it follows by Remark \ref{surj} that $H^1(M;\Z_2)\cong \Z_2$ and hence $H^*(M;\Z_2)\cong \Z_2[x]/(x^{n+1})$, with $\deg(x)=1$, so $M$ is a $\Z_2$-cohomology $\RP^{n}$. 
We assume then that $H^*(M;\Z_2)$ is $2$-periodic, but not $1$-periodic. We break the reminder of the proof of Part 1 into two cases, according to whether the dimension of $M$ is even or odd.

 \vspace{.05cm}
\noindent{\bf Part 1, Case 1, when ${\mathbf{n}}$ is odd:} It follows from $2$-periodicity and Poincar\'e duality 
 that $H^i(M;\Z_2)\cong\Z_2$ for $i\in\{1,\cdots, n\}$. Again by $2$-periodicity, we have that $H^*(M;\Z_2)$ is generated by two elements, namely, $x\in H^1(M;\Z_2)$ and $y\in H^2(M;\Z_2)$. If $x^2=y$, then  Lemma \ref{factor} implies that $H^*(M;\Z_2)$ is $1$-periodic. Since we assumed that $H^*(M;\Z_2)$ is not $1$-periodic, $x^2=0$. It follows that $H^*(M;\Z_2)\cong \Z_2[x,y]/\(x^2,y^{(n+1)/2}\), \deg(x) =1, \deg(y)=2$. 
 Recall from Lemma \ref{cohomlens} that this is the $\Z_2$-cohomology ring of a lens space $L^n_m = S^n/\Z_m$ for $m \equiv 0\pmod{4}$. 

 \vspace{.05cm}
\noindent{\bf Part 2, Case 2, when ${\mathbf{n}}$ is even:} It follows from $2$-periodicity that the $\Z_2$-cohomology of $M$ is of the form
		\[ 
			H^i(M;\Z_2)\cong 
				\left\{\begin{array}{ll}  
								\Z^j_2 & \text{ if } i\,\,\, \text{is odd}\\
								\Z_2 & \text{ if } i\,\,\, \text{is even} \						
				 \end{array}\right.,
		\]
\noindent where $j$ is a fixed non-negative integer.  We show that $j=0$, that is, $M$ has the $\Z_2$-cohomology of $\CP^{\frac{n}2}$. Let $x_i$ denote a non-trivial element in $H^i(M; \Z_2)$. If instead $j\geq 1$, it follows from Poincar\'e duality that $H^{n}(M;\Z_2)\cong \Z_2$ is generated by a product of non-trivial elements $x_{\scriptstyle 1}\smile 
x_{\scriptstyle {n-1}}$. But by $2$-periodicity we know that $H^{n}(M;\Z_2)\cong \Z_2$ is generated by 
$(x_{\scriptstyle 2})^{n/2}$, where $x_{\scriptstyle 2}$ is the periodicity-inducing element in $H^{*}(M;\Z_2)$. Therefore, $x_{\scriptstyle 1}\smile x_{\scriptstyle {n-1}}=(x_{\scriptstyle 2})^{n/2}$, and it follows from Lemma \ref{factor} that $x_1$ is a periodicity-inducing element of degree $1$
in $H^*(M;\Z_2)$, a contradiction to   the assumption that $H^*(M;\Z_2)$ is not $1$-periodic.  
This completes the proof of Part 1.

We now prove Part 2. Since $H^*(M;\Z_2)$ is assumed to be 4-periodic, the minimal degree of periodicity in $H^*(M;\Z_2)$ is 1, 2 or 4. If the minimal degree of periodicity is 1 or 2, then Part 1 applies.
Hence we assume that $4$ is the minimal degree of periodicity in $H^*(M;\Z_2)$ and that the 4-periodicity inducing element, $x\in H^4(M;\Z_2)$, is non-trivial.
We break the remainder of the proof into cases according to the congruence class of $n\pmod{4}$. 

 \vspace{.05cm}
\noindent{\bf Part 2, Case 1, when $\mathbf{n\equiv 0\,\,(mod \,4)}$}: 
By hypothesis, $x^{n/4}$ generates $H^{n}(M;\Z_2)$. It follows by Poincar\'e duality and 
Lemma \ref{factor} that any non-trivial $y\in H^i(M;\Z_2)$ with $1\leq i\leq 3$ is a periodicity-inducing element of degree $i$, a contradiction to 
$4$ being the minimal degree of periodicity.
Hence $y=0$ and $M\sim_{\Z_2}\HP^{\frac{n}4}$.

 \vspace{.05cm}
\noindent{\bf Part 2, Case 2, when $\mathbf{n\equiv 1\,\,(mod \,4)}$}: We show that in this case $H^*(M;\Z_2)$ must have minimal periodicity of degree less than four, hence Part 1 applies.  By Proposition \ref{4z2}, 
$\widetilde{M}$ must be a $\Z_2$-cohomology sphere. 
 Since $G=\pi_1(M)$ acts freely on a $\Z_2$-cohomology sphere, Proposition \ref{zpzp} implies that $\Z_p^2$ is not a subgroup of $G$  for any prime $p$. By Theorem \ref{syl} and Remark 5.6 in Adem \cite{Ad}, it follows that $H^*(G;\Z_2)$ is periodic with strictly positive period dividing $(n+1)$. 
Since $(n+1)\equiv2\pmod 4$, Theorem \ref{z2} implies that the minimal degree of periodicity of $H^*(G;\Z_2) \cong H^*(BG;\Z_2)$ is $1$ or $2$. In the Serre spectral sequence associated to the Borel fibration $\widetilde{M}\to EG \times_{G}\widetilde{M}\ \to BG$, $d_{n+1}$ is the only non-trivial differential, since $\widetilde{M}$ is a $\Z_2$-cohomology sphere.  It follows that $E_3=E_{\infty}$ and so  $H^j(M;\Z_2) \cong H^j(\widetilde{M}/G;\Z_2)\cong E_{3}^{j,0}\cong H^j(BG;\Z_2)$ for $0\leq j\leq n$. We conclude that $H^*(M;\Z_2)$ is  $1$- or $2$- periodic, respectively.

 \vspace{.05cm}
\vspace{.05cm}
\noindent{\bf Part 2, Case 3, when $\mathbf{n\equiv 2\,\,(mod \,4)}$}:
By Synge's Theorem, $\pi_1(M)$ is either 0 or $\Z_2$. 
 The simply connected case follows from Proposition \ref{4z2}, so we only consider $\pi_1(M)\cong \Z_2$. If $\widetilde{M}\sim_{\Z_2}S^{n}$ then Theorem \ref{Su} implies that $M\sim_{\Z_2}\RP^n$. If $\widetilde{M}\sim_{\Z_2}\CP^{n/2}$ or $S^2\times \HP^{\frac{n-2}4}$ then, we again compute the Serre spectral sequence of the Borel fibration $\widetilde{M}\to E\Z_2 \times_{\\Z_2}\widetilde{M}\to B\Z_2$. If $y$ generates $H^2(\widetilde{M};\Z_2)$, the following proof works regardless of whether $y^2$ is trivial or non-trivial, that is, we may argue the cases of $\widetilde{M}\sim_{\Z_2}\CP^{n/2}$ and $S^2\times \HP^{\frac{n-2}4}$ simultaneously.
On the $E_2$ page, we see $d_2$ is trivial and hence $E_2=E_3$.  On the $E_3$-page, $d_3:E_3^{0,2}\cong \Z_2 \to E_3^{3,0} \cong \Z_2$ is either trivial or an isomorphism.  If it is trivial, then both $E_3^{0,2}$ and $E_3^{2,0}$ survive to $E_{\infty}$, and $H^2(M;\Z_2)\cong (\Z_2)^2$.  However,
 Poincar\'e duality and $4$-periodicity imply that $H^2(M;\Z_2)\cong H^2(\widetilde{M};\Z_2)\cong\Z_2$.  Hence $d_3:E_3^{0,2}\to E_3^{3,0}$ is an isomorphism.  Using the fact that $d_3\circ d_3=0$ we obtain that
 $d_3:E_3^{0,4}\to E_3^{3,2}$ is trivial. The product structure on $E_3$ then implies that $E_4^{i,4l}\cong \Z_2$ for $i\leq 2$ and $l\leq\frac{n+1}4$, and $E^{i,j}_4\cong0$ otherwise. Then $d_4=0$, $E_4=E_{\infty}$, and we see from the product structure on $E_4$ that 
$M\sim_{\Z_2}\RP^2\times\HP^{\frac{n-2}4}$.

 \vspace{.05cm}
\noindent {\bf Part 2, Case 4, when $\mathbf{n\equiv 3\,\,(mod \,4)}$}:  
Since the periodicity-inducing element $x_4\in H^4(M;\Z_2)$ is non-trivial, it follows from Poincar\'e duality and $4$-periodicity that $H^3\(M;\Z_2\)\cong H^4\(M;\Z_2\)\cong \Z_2$ and $H^1\(M;\Z_2\)\cong H^2\(M;\Z_2\)\cong \Z^l_2$ for some $l\geq 0$. 

If $l=0$ then $M\sim_{\Z_2}S^3\times \HP^{\frac{n-3}4}$ and the result follows. 
 Suppose now that $l=1$. Let $x_i$ denote the generator of $H^i(M;\Z_2)\cong \Z_2$ for $0\leq i\leq n$. Since 4 is the minimal degree of periodicity in $H^*(M;\Z_2)$, using Lemma \ref{factor}, 
 we see that  $x_i\smile x_j = 0$ for $(i, j)\in\{(1, 3), (2, 2)\}$. We now claim that $x_i\smile x_j =0$ for $(i, j)\in\{(2, 3), (3, 3)\}$. Since we are using $\Z_2$-coefficients, the dual of the generator of $H_k(M; \Z_2)$ is isomorphic to $x_k$. If $x_2\smile x_3$ is non-trivial, then  
 by Poincar\'e duality $x_{n-5}\smile x_2 \smile x_3 =x_n.$
  However, $x_{n-5}=x_2\smile x_4^{(n-7)/4}$, and  by properties of the cup product, we see that $x_n=x_2\smile x_4^{(n-7)/4}\smile x_2 \smile x_3 =0$, since $x^2=0$, a contradiction.
 The argument for $x_3\smile x_3$ is similar and the claim holds. 

 Likewise, we claim that $x_1\smile x_2$ is non-trivial.
By Poincar\'e duality $x_1\smile x_2\smile x_{n-3} = x_n$.  However, $x_{n-3}$ is isomorphic by $4$-periodicity to $x_4^{(n-3)/4}$ and since this element is non-trivial, $x_1\smile x_2$ must be, as well.

 Using $4$-periodicity, $H^*(M;\Z_2)$ is determined once we know whether $x_1^2$ is trivial or not. If $x_1^2$ is trivial, then $M\sim_{\Z_2} L_4^3\times \HP^{\frac{n-3}4}$, and if $x_1^2$ is non-trivial then $M\sim_{\Z_2} \RP^3\times \HP^{\frac{n-3}4}$, and the result follows.

 We now consider the case where $l\geq 2$. 
 Lemma \ref{4z2} implies that $\widetilde{M}\sim_{\Z_2} S^{n}$ or $S^3\times\HP^{\frac{n-3}4}$.
 We first show that $\Z_2^2$ is not a subgroup of $\pi_1(M)$. If $\widetilde{M}\sim_{\Z_2} S^{n}$, then this follows by Corollary \ref{zpzp}. If  instead $\widetilde{M}\sim_{\Z_2} S^3\times\HP^{\frac{n-3}4}$, then since
 we assume that any cover of $M$ must have $4$-periodic $\Z_2$-cohomology,
 and this follows by Lemma \ref{freez2}.

Corollary \ref{sylc} then gives us that $Syl_2(\pi_1(M))$ is either $\Z_{2^r}$ or $Q_{2^r}$.  
We let $Syl_2(\pi_1(M))=\Gamma$, $\widetilde{M}/\Gamma=\widebar{M}$, and consider the composition of covering maps
$\widetilde{M}\to\widebar{M}\xrightarrow{p} M.$
Since $\Gamma$ is a maximal $2$-subgroup of $\pi_1(M)$, it follows that $p:\widebar{M}\to M$ is an odd-sheeted covering and Proposition \ref{transfer} gives us that $p^*: H^i(M;\Z_2)\to H^i\(\widebar{M};\Z_2\)$ is injective for all $i$, and so  
$H^1(\widebar{M};\Z_2)$ also has rank $\geq 2$. 
Calculating the Serre spectral sequence associated to the Borel fibration $\widetilde{M}\to \widetilde{M}\times_{\Gamma}E\Gamma\to B\Gamma$, we obtain that $d_m:E_m^{1-m,m-1}\to E_m^{1,0}$ is trivial for all $m\geq 2$. It follows that $E_2^{1,0}\cong E_{\infty}^{1,0}$, and $H^1(\widebar{M};\Z_2) \cong H^1(\Gamma;\Z_2)$. Since the rank of $H^1(\Gamma;\Z_2)$ is then greater or equal to  $2$, Proposition \ref{cohomQ} shows that $\Gamma$ must be generalized quaternion. It follows that $\widebar{M}$ has the $\Z_2$-cohomology of either $S^n/Q_i$ or $(S^3\times\HP^{\frac{n-3}4})/Q_i$ for $i\in\{8,16\}$.

Finally, we show that $S^n/Q_i\sim_{\Z_2}(S^3\times\HP^{\frac{n-3}4})/Q_i$ for $i\in\{8,16\}$. Suppose that $\widetilde{M}\sim_{\Z_2}S^n$. Computing the Serre spectral sequence associated to the Borel fibration $\widetilde{M}\to \widetilde{M}\times_{Q_i}E\Gamma\to BQ_i$, we see that $d_{n+1}$ is the only non-trivial differential and that $H^j(\widebar{M};\Z_2)\cong H^j(Q_i;\Z_2)$ for $0\leq j\leq n, i\in\{8,16\}$, where the ring structure of $H^*(\widebar{M};\Z_2)$ is identical to the ring structure of $H^*(Q_i;\Z_2)$, which is described in Lemma \ref{cohomQ}.
The K\"unneth formula now implies that $S^n/Q_i\sim_{\Z_2}S^3/Q_i\times\HP^{\frac{n-3}4}$. 

It remains to show that $(S^3\times\HP^{\frac{n-3}4})/Q_i\sim_{\Z_2}S^3/Q_i\times\HP^{\frac{n-3}4}$.
Suppose that $\widetilde{M}\sim_{\Z_2}S^3\times\HP^{\frac{n-3}4}$. Computing the Serre spectral sequence associated to the Borel fibration $\widetilde{M}\to \widetilde{M}\times_{Q_i}E\Gamma\to BQ_i$, we see that $d_{4}$ is the first potentially non-trivial differential.
 Since $H^*(\widetilde{M}/Q_i;\Z_2)$ is 4-periodic, we have that $H^3(\widetilde{M}/Q_i;\Z_2)\cong \Z_2$.   Observing that $E^{3,0}_{4}\cong E^{0,3}_4\cong\Z_2$, we conclude that one of these groups must vanish at the infinity page. Since $E^{3,0}_j\cong\Z_2$ is never in the image of a non-trivial differential, we have $E^{3,0}_{\infty}\cong\Z_2$. Hence $E^{0,3}_{\infty}\cong 0$ which implies that  $d_4:E_4^{0,3}\to E_4^{4,0}$ is an isomorphism.  Using the product structure on $E_4$ and the periodicity of
 $H^*(\widetilde{M};\Z_2)$ and $H^*(Q_i;\Z_2)$, we can completely determine $d_4$ and compute $E_5$.  We find that $E_5^{l,0}\cong H^l(Q_i;\Z_2)$ for $0\leq l\leq 3$, $E_5^{l,4j}\cong E_5^{l,0}$ for $0\leq j\leq \lfloor\frac{n}4 \rfloor$ and $0\leq l \leq 3$, and all other groups on the $E_5$ page are trivial. Observing that $d_5$ is trivial on the $E_4$ page, we have that $E_4=E_{\infty}$, and the product structure on $E_4$ gives us that $(S^3\times\HP^{\frac{n-3}4})/Q_i\sim_{\Z_2}S^3/Q_i\times\HP^{\frac{n-3}4}$.

Recall from the calculation of $H^*(\widebar{M};\Z_2)$, we saw that $l=2$. Since  $p*:H^*(M;\Z_2)\to H^*(\widebar{M};\Z_2)$ is injective, this means that $p_*$ is, in fact, an isomorphism. This completes the proof. 
\end{proof}

\subsection{Fixed Point Sets of $\Z_2$-Actions}
We begin by  stating Theorem VII.3.2 of Bredon \cite{Br}.

\begin{theorem}\textup{\cite{Br}}\label{Bredon}
Let $M$ be a closed manifold with the $\Z_2$-cohomology of $\FP^{n}$ with $n\geq 2$, where $\F$ is one of $\R$, $\C$, or $\mathbb{H}$. If $\Z_2$ acts effectively on $M$ with non-empty fixed-point set $F$, then one of the following occurs:
\begin{enumerate}
\item $F$ is connected and a $\Z_2$-cohomology $\RP^n$, and $M$ is a $\Z_2$-cohomology $\CP^n$.
\item $F$ is connected and a $\Z_2$-cohomology $\CP^n$, and $M$ is a $\Z_2$-cohomology $\HP^n$.
\item $F$ has two components, $F_1$ and $F_2$, and if $M$ is a $\Z_2$-cohomology $\FP^n$, then each $F_i$ is a $\Z_2$-cohomology $\FP^{n_i}$ such that $n=n_1+n_2+1$, where $n_i\geq 0$.
\end{enumerate}
\end{theorem}

For actions by a finite $p$-group on a manifold $M$, Theorem III.5.1 of \cite{Br} determines the homology groups of the fixed point set, $M^G$, when $M$ is a $\Z_p$-homology sphere, and Theorem III.4.3 of \cite{Br} relates the Euler characteristic of the fixed-point set to the Euler characteristic of $M$. We summarize these two theorems for the case when $p=2$ in the following theorem.

\begin{theorem}\textup{\cite{Br}}\label{smith}. 
 Let $G$ be a finite $2$-group, and let $X$ be a finite-dimensional simplicial $G$-complex. The following then hold:
\begin{enumerate}
\item If all the $\Z_2$-homology groups of $X$ are finite dimensional, then the same is true for $X^G$ and $\euc(X)\equiv\euc(X^G)\pmod{p}$.
\item If $X$ is a $\Z_2$-homology n-sphere, then $X^G$ is either empty or a $\Z_2$-homology $m$-sphere for some $0\leq m\leq n$.
\end{enumerate}
\end{theorem}

We now recall the Borel formula, which is a fundamental tool in the proof of both Theorem \hyperref[TA]{A} and Theorem \hyperref[TB]{B}.

\begin{theorem}[The Borel Formula]\textup{\cite{borel}}\label{borel}.
Let $\Z_2^r$ act smoothly on $M$, a Poincar\'e duality space, with fixed-point set component $F$. Then
		\[\mathrm{codim}(F\subseteq M)=\sum\mathrm{codim}\(F\subseteq F'\),\]	
\noindent where the sum runs over the fixed-point set components $F'$ of corank-1 subgroups $\Z_2^{r-1}\leq \Z_2^r$ for which $\dim(F')>\dim(F)$.

	These subgroups are precisely the kernels of the irreducible subrepresentations of the isotropy representation at a normal space to $F$. In particular, the number of pairwise-inequivalent irreducible subrepresentations of the isotropy representation at a normal space to $F$ is at least $r$ if the action is effective, and equality holds only if the isotropy representation is equivalent to one of the form
		\[(\epsilon_1,\cdots,\epsilon_r)\mapsto (\epsilon_1I_{m_1},\cdots, \epsilon_rI_{m_r}),\]	
\noindent where $\epsilon_i\in\{\pm1\}$ and $m_i>0$, and $I_m$ denotes the identity matrix of rank $m$. 	
\end{theorem}

\subsection{Positive sectional curvature}
In this section we collect known results on positively curved manifolds.
We begin by recalling Part 1 of the \hyperref[connlemma] {Connectedness Lemma} of \cite{Wi}. 

\begin{conn}\textup{\cite{Wi}} \label{connlemma}
Let $M^{m+k}$ be a compact Riemannian manifold with positive sectional curvature.
Suppose $N^{m}\subset M^{m+k}$ is a compact totally geodesic embedded submanifold of codimension $k$. Then the inclusion map $N\hookrightarrow M$ is $(m-k+1)$-connected. 
\end{conn}

Using the \hyperref[connlemma]{Connectedness Lemma},  in \cite{KKSS} they prove a set of Codimension Lemmas  for  closed, totally geodesic, maximally connected submanifolds of codimensions $1-4$ of a closed, positively curved Riemannian manifold. We combine all four Codimension Lemmas in the following proposition.

\begin{proposition}\label{codimsmall}\textup{\cite{KKSS}}
Let $M^{m+k}$ be a closed, positively curved Riemannian manifold with $1\leq k\leq 4$. Let $N^m$ be a closed, totally geodesic, $m$-connected submanifold of $M$ with $m\geq k$  if $3\leq k\leq 4$. Then one of the following holds:
\begin{enumerate}
\item $k$ is odd and $N$ is homotopy equivalent to one of $S^m$ or $\RP^m$; 
\item $k=2$ and  $N$ is homotopy equivalent to one of $S^m$, $\RP^m$, $\CP^{\frac{m}2}$, or a lens space; or 
\item $k=4$ and one of the following occurs
\begin{enumerate}
\item 
  $N$ is homotopy equivalent to one of $S^m/\Z_k$, $k\geq 1$, $\CP^{\frac{m}2}$, or is an integral cohomology  $\HP^{\frac{m}{4}}$; 
 
\item  $\widetilde{N}$  has  the integral cohomology of $E^{m}_{\ell}$, $m\equiv 2\mod 4$;
\item $\widetilde{N}$  has  the integral cohomology of  $S^3\times \HP^{\frac{m-3}{4}}$ or $N^{m}_j$, $m\equiv 3\mod 4$.
\end{enumerate}
\end{enumerate}
Moreover, if $k=1$, or $k=2$ and $m$ is odd, the assumption that $N$ is an $m$-connected submanifold of $M$ is not necessary. 
\end{proposition}

The following observation from \cite{KKSS} is a direct consequence of Theorem \ref{borel} and can be used to verify that a fixed point set of a group action is maximally connected.
\begin{observation}\label{connrep}\textup{\cite{KKSS}}  Let $F_j$ be the fixed-point set of a $\Z_2^{r-j}$-action by isometries on
a closed, positively curved manifold $M$, and suppose that $r - j \geq 2$. If the isotropy representation has exactly $r-j$ irreducible subrepresentations, and if the fixed-point set component $F_{j+1}$ containing $F_j$ of some $\Z_2^{r-j-1}$ has the property that the codimension $k_j$ of the inclusion $F_j \subseteq F_{j+1}$ is minimal, then this inclusion is $\dim(F_j)$-connected.
\end{observation}

Finally, 
we recall a useful fundamental group result of Frank, Rong, and Wang \cite{FRW}.

\begin{lemma}\textup{\cite{FRW}}\label{codim2cyclic}
Let  $M$ be a closed $n$-manifold of positive sectional curvature with $n\geq 5$. If $M$ has a totally geodesic submanifold $N$ of codimension 2, then $\pi_1(M)$ is cyclic.
\end{lemma}

\section{The Higher Codimension Lemmas}\label{3}

 In this section we apply the \hyperref[connlemma]{Connectedness Lemma} and \hyperref[periodicity]{Periodicity Lemma} to obtain results for maximally connected submanifolds of codimensions greater than $4$. We first prove
 a lemma for odd codimensions that generalizes Part 1 of Proposition \ref{codimsmall} to all odd codimensions. We then prove a lemma for codimensions greater than or equal to $6$ and congruent to $2\pmod 4$, where instead of being able to classify the homotopy type of the maximally connected submanifold, we classify its $\Z_2$-cohomology ring. Lastly, we prove a lemma for all codimensions congruent to $4\pmod 8$ which again  only gives us the $\Z_2$-cohomology ring of the maximally connected submanifold.

 \begin{oddcodimlemma} \label{odd} 
Let $k$ be an odd integer. Suppose $M^{m+k}$ is a positively curved, closed Riemannian manifold with $N^m\subset M^{m+k}$ a totally-geodesic submanifold. Suppose $m\geq k$ and the inclusion $N\hookrightarrow M$ is $m$-connected. Then $N$ is homotopy equivalent to either $S^m$ or $\RP^m$.
\end{oddcodimlemma}

\begin{proof} 
When $k=1$ or $k=3$, the \hyperref[odd]{Odd Codimension Lemma} follows from Proposition \ref{codimsmall}. 
We assume now that $k\geq5$ be odd, $m\geq k$.  
It follows from the $m$-connectivity of $N\hookrightarrow M$ that $\pi_1(N)\cong\pi_1(M)$. Since one of $\dim(N)$ or $\dim(M)$ is even, as $k$ is odd, by Synge's theorem $\pi_1(N)$ is either 0 or $\Z_2$. Our goal is to show that the universal cover of $N$ is a $\Z$-cohomology sphere, and  the result then follows from Theorem \ref{universal}. 
 
Consider the universal cover $p:\widetilde{M}\to M$. By Lemma \ref{connlifting}, $p^{-1}(N)=\widetilde{N}$ is the universal cover of $N$ and $\widetilde{N}\hookrightarrow\widetilde{M}$ is $m$-connected. It follows from the \hyperref[periodicity]{Periodicity Lemma} that $H^*(\widetilde{M};\Z)$ is $k$-periodic. 
 Remark \ref{surj} gives us that $H^{k}(\widetilde{M};\Z)$ is singly-generated. Since $m\geq k$ and $\codim(N\subseteq M)$ is odd, arguing as in the proof of Proposition 7.5 in \cite{Wi}, we see that $H^{k}(\widetilde{M};\Z)$ is either $0$ or $\Z_2$. 

We claim that $H^k(\widetilde{M};\Z)$ is trivial.
 The \hyperref[periodicity]{Periodicity Lemma} implies that $H^*(\widetilde{M};\Z_2)$ is $k$-periodic.
  Let  $x\in H^l(\widetilde{M};\Z_2)$ be the non-trivial element of minimal degree inducing $k$-periodicity. By Theorem \ref{z2}, $l$ is a power of 2 and divides $k$.  Hence $l=1$ and $H^*(\widetilde{M};\Z_2)$ is $1$-periodic. But $\widetilde{M}$ is simply connected and so $H^k(\widetilde{M};\Z_2)$ must also be trivial.
From the Bockstein sequence associated to $\Z\xrightarrow{\times 2}\Z\to\Z_2$ at degree $k$ we obtain
		\[H^k(\widetilde{M};\Z)\xrightarrow{\times 2}H^k(\widetilde{M};\Z)\to H^k(\widetilde{M};\Z_2)\cong 0.\]
\noindent By exactness we conclude that $H^k(\widetilde{M};\Z)\ncong\Z_2$. Hence $H^k(\widetilde{M};\Z)\cong 0$ and the periodicity-inducing element is also trivial.
Thus $\widetilde{M}$ is a $\Z$-cohomology sphere. 
It follows from the $m$-connectivity of $\widetilde{N}\hookrightarrow\widetilde{M}$ that $\widetilde{N}$ is a $\Z$-cohomology sphere, and 
this concludes the proof.
\end{proof}

\begin{codim2mod4}\label{2mod4}
Let $k\equiv 2\pmod{4}$, with $k\geq 6$.  Suppose $M^{m+k}$ is a closed, positively curved Riemannian manifold with $N^m$ a closed, totally geodesic submanifold of $M$. If $m\geq k$ and the inclusion $N^m\hookrightarrow M^{m+k}$ is $m$-connected, then $N$ has the $\Z_2$-cohomology of $S^m$, $\CP^{m/2}$, $\RP^m$, or a lens space $L_4^m$. Moreover, the same conclusions hold for $M$ with the appropriate modifications to the dimensions.
\end{codim2mod4}

\begin{proof} 
Since $N\hookrightarrow M$ is $m$-connected and $N$ and $M$ are both $\Z_2$-orientable, it follows from the \hyperref[periodicity]{Periodicity Lemma} that $H^*(M;\Z_2)$ is $k$-periodic. Since $4\not\divides k$, by Theorem \ref{z2}, $H^*(M;\Z_2)$ is $2$-periodic. Hence by Theorem \ref{z2anypi1}, $M$ has the $\Z_2$-cohomology of $S^m$, $\CP^{m/2}$, $\RP^m$, or $L_4^m$.  Since $N\hookrightarrow M$ is $m$-connected, we can conclude the same for $N$ as well, completing the proof. 
\end{proof}

\begin{mod2codim4lemma}\label{z2codim4}
Let $k\equiv 4\pmod 8$. Suppose $M^{m+k}$ is a closed, positively curved Riemannian manifold with $N^m$ a closed, totally geodesic submanifold of $M$. If $m\geq k$ and the inclusion $N^m\hookrightarrow M^{m+k}$ is $m$-connected,  then $N$ has
the $\Z_2$-cohomology of $S^{m}/\Z_k$ for $k\in\{1,2,4\}$, $\CP^{m/2}$, $\HP^{m/4}$, $S^2\times \HP^{\frac{m-2}4}$, $\RP^2\times \HP^{\frac{m-2}4}$, $S^3/\Z_k\times\HP^{\frac{m-3}4}$ with $k\in\{1,2,4\}$, or $S^{m}/Q_i$ for $i\in\{8,16\}$ and $m\equiv 3\pmod{4}$.
Moreover, the same conclusions hold for $M$ with the appropriate modifications to the dimensions.
\end{mod2codim4lemma}

\begin{proof}
Since the inclusion $N\hookrightarrow M$ is $m$-connected, it follows from Proposition \ref{connlifting} that $\widetilde{N}\hookrightarrow \widetilde{M}$ is also $m$-connected. The \hyperref[periodicity]{Periodicity Lemma} and Lemma \ref{connlifting} gives us that that $H^*(M;\Z_2)$ and $H^*(\widebar{M};\Z_2)$ are $k$-periodic for any cover $\widebar{M}\to M$. Since $8\nmid k$,  by Theorem \ref{z2} we know that $H^*(M;\Z_2)$ and $H^*(\widebar{M};\Z_2)$ are both  $4$-periodic. The result now follows by Theorem \ref{z2anypi1}.
\end{proof}

\section{Restrictions on Chains of Fixed-Point Set Components}\label{4}

Our goal in this section is to establish the tools we need to prove Theorems \hyperref[TA]{A} and \hyperref[TB]{B}. 
We first recall the following observation mentioned in the Introduction that appears in Su and Wang \cite{SW}.

\begin{proposition}\cite{SW}\label{chain} 
 Let $M$ be a closed, positively curved manifold with an isometric, effective $\Z_2^r$-action with a fixed point. Then there is an ascending chain of fixed-point set components $F_j\subseteq F_{j+1}$ so that 
the effective kernel of the induced $\Z_2^r$-action on $F_j$ is $\Z_2^{r-j}$ and  
for each $j$,  $k_j\leq k_{j+1}$. 
\end{proposition}

The following corollary combines the fact that for each $j$,  $k_j\leq k_{j+1}$, with a condition on the maximal connectedness of the inclusions $F_j \hookrightarrow F_{j+1} $, see 
Observation \ref{connrep}. Recall that $\dim(F_j)=m_j$.

\begin{corollary}\label{chain!}
Let $M$ be a positively curved closed manifold with an isometric, effective $\Z_2^r$-action with a fixed point. If $k_{j+1}< 2k_{j}$ in the ascending chain of fixed-point set components 
 $F_j\subseteq F_{j+1}$, then the inclusion $F_j\hookrightarrow F_{j+1}$ is $m_j$-connected.
\end{corollary}

\begin{proof}
   By Proposition \ref{chain}, we have a chain of fixed-point set components.  Consider the submanifolds $F_{j}\hookrightarrow F_{j+1}\hookrightarrow F_{j+2}$. Since the effective kernel of the induced $\Z_2^r$-action on $F_i$ is $\Z_2^{r-i}$ for all $i$, we can find a $\Z_2^2\subseteq\Z_2^r$ so that $\Z_2^2$ acts effectively and isometrically on  $F_{j+2}$ with fixed-point set $F_{j}$ and an involution fixing  $F_{j+1}$. Applying Theorem \ref{borel}, and the fact that $k_j\leq k_{j+1}$, we see that if $k_{j+1}< 2k_{j}$, then there are exactly two pairwise inequivalent irreducible subrepresentations of $\Z_2^2$ at the normal space to $F_{j}$ in $F_{j+2}$. It follows from Observation \ref{connrep} that the inclusion $F_j\hookrightarrow F_{j+1}$ is $m_j$-connected.
\end{proof}

 In the following lemma, we show that if $r$ is bounded below by approximately $n/b$, $b\in \N$, and both $k_0$ and $m_0$ are greater than or equal to $2$, then the codimensions of the inclusions, $F_i\subset F_{i+1}$, are at most $b-1$ for the first $\lfloor b/2 \rfloor -1$ values of $i$. 

\begin{lemma}\label{bounds}
    Let $M^n$ be an $n$-manifold admitting an isometric, effective $\Z_2^r$-action. Suppose there exists a chain of fixed-point sets as in Proposition \ref{chain} with $k_0\geq 2$, $m_0\geq 3$, 
    and $r\geq \frac{n}b+\frac{b}2-2.$ Then $2\leq k_{j}\leq b-1$ for $0\leq j\leq \lfloor \frac{b}{2}\rfloor-1$. 
    \end{lemma}

\begin{proof}
   To prove the lemma, 
   we show that if for some $l$ we have $k_l\geq b$, then $l\geq \lfloor\frac{b}{2}\rfloor$. In particular, this implies that $k_{l-1}$, and hence $k_j$ for $0\leq j\leq l-1$ must be bounded between $2$ and $b-1$.

Suppose that for some $l$ we have $k_l\geq b$. Since $r\geq\frac{n}{b}+\frac{b}2-2$, $k_i\geq 2$ for all $i$, and $m_0\geq 3$, it follows that
$$\label{Inequality3}
n 
=m_{0}+\sum_{j=0}^{r-1}k_j\geq m_0+2l+ \sum_{j=l}^{r-1}k_j\geq 3+2l+ b(r-l) \ge
	 \(\frac{b^2-4b+4}{2}\)-(b-2)l+1 +n.$$
So we have $$l\geq \frac{b}{2} + \frac{1}{b-2}-1 \ge \left\lfloor\frac{b}{2}\right\rfloor + \frac{1}{b-2}-1 ,$$
and the result follows.
\end{proof}

We note that in the situation of Corollary \ref{chain!}, where the inclusion $F_j\hookrightarrow F_{j+1}$ is maximally connected, in the case that $1\leq k_j\leq 5$ and $m_j\geq k_j$ for $k_j\geq 3$, we may apply either Proposition \ref{codimsmall} or the \hyperref[odd]{Odd Codimension Lemma}, to identify  the integral cohomology ring of  $F_j$ or classify $F_j$ up to homotopy equivalence. We summarize this in the following lemma, which is used in the proof of Theorem \hyperref[TA]{A}. 

\begin{lemma}\label{chain2} Let $F_j\hookrightarrow F_{j+1}$ be part of a chain of closed, totally geodesic submanifolds in a closed, positively curved manifold, $M$, as defined in Proposition \ref{chain}. 
Suppose for $k_j\geq 2$ that $k_{j+1}\leq  2k_j-1$ except when $k_j=2$ and $m_j$ is odd,
and  when $k_j\geq 3$, assume additionally that $m_j\geq k_j$, 
then the following hold:
\begin{enumerate}
    \item If $k_j$ is odd,
$F_j$ is 
homotopy equivalent to either $S^{m_j}$ or $\RP^{m_j}$. 

\item If $k_j=2$, and  $F_j$ is 
homotopy equivalent  to one of $S^{m_j}$, $\RP^{m_j}$, $\CP^{m_j/2}$, or a lens space. 
\item If $k_j=4$, then

\begin{enumerate}
\item 
  $F_j$ is homotopy equivalent to one of $S^{m_j}/\Z_k$, $k\geq 1$, $\CP^{m_j/2}$, or is an integral cohomology  $\HP^{m_j/4}$; 
 
\item  $\widetilde{F}_j$  has  the integral cohomology of $E^{m_j}_{\ell}$, $m\equiv 2\mod 4$;
\item $\widetilde{F}_j$  has  the integral cohomology of  $S^3\times \HP^{(m_j-3)/4}$ or $N^{m}_j$, $m\equiv 3\mod 4$.
\end{enumerate}
\end{enumerate}
\end{lemma}

Combining Corollary \ref{chain!} with the \hyperref[z2codim4]{$4\pmod 8$ Codimension Lemma} and the \hyperref[2mod4]{$2 \pmod 4$ Codimension Lemma}, we obtain the next corollary, which we need for the proof of Theorem \hyperref[TB]{B}.
\begin{lemma}\label{chain3} Let $F_j\hookrightarrow F_{j+1}$ be part of a chain of closed, totally geodesic submanifolds in a closed, positively curved manifold, $M$, as defined in Proposition \ref{chain}. 
Suppose $k_{j+1}\leq  2k_j-1$
and  $m_j\geq k_j$, then the following hold:
\begin{enumerate}
 
\item If $k_j=4$, then $F_j$ has the $\Z_2$-cohomology of $S^{m_j}/\Z_k$ for $k\in\{1,2,4\}$, $\CP^{m_j/2}$, $\HP^{m_j/4}$, $S^2\times \HP^{(m_j-2)/4}$, $\RP^2\times \HP^{(m_j-2)/4}$, $S^3/\Z_k\times\HP^{(m_j-3)/4}$ with $k\in\{1,2,4\}$, or $S^{m_j}/Q_i$ for $i\in\{8,16\}$ and $m_j\equiv 3\pmod{4}$.

\item If $k_j=6$, then $F_j$ has the $\Z_2$-cohomology of $S^{m_j}$, $\RP^{m_j}$, $\CP^{m_j}$, or a lens space.
\end{enumerate}
\end{lemma}

We note that in the proofs of Theorems \hyperref[TA]{A} and \hyperref[TB]{B}, our goal is to apply Lemma \ref{chain2} or Lemmas \ref{chain2} and  \ref{chain3}, respectively. To apply either one to $F_0$, we need $m_0\geq k_0$ and the inclusion of $F_0$ in $F_1$ to be maximally connected. 
In the following lemma, we consider the case where   $m_0< k_0$ is a possibility. 

\begin{proposition}\label{l1}

Let $M$ be a closed, positively curved manifold, and $F_i$ be closed, totally geodesic submanifolds of $M$ of dimensions $m_i$ satisfying
        \[F^{}_0\subseteq F^{}_1\subseteq F^{}_2 \subseteq F^{}_3,\]
with $m_1\geq k_1$, $k_{i+1}\leq 2k_{i}-1$ for $i\in\{1, 2\}$, with at least one of the $k_i$, $i\in\{1, 2\}$, is odd. Suppose further that $k_i\leq 5$, $0\leq i\leq 2$.  If either
\begin{enumerate}
\item $k_1\leq 2k_0-1$; or
\item $k_1\geq 2k_0$ and $F_0 \hookrightarrow F_1$ is at least $(\lfloor \frac{m_0}{2}\rfloor +1)$-connected.
\end{enumerate}
Then $F_0$ is homotopy equivalent to a sphere or a real projective space.
\end{proposition}

\begin{proof}

Using Corollary \ref{chain!}, it follows that 
$F_i$ is maximally connected in $F_{i+1}$ for $i\in\{1, 2\}$. By hypothesis, $m_1\geq k_1$, and since $k_2\leq 2k_1-1$, we have  that $m_2\geq k_2$. Combining the hypothesis that   $k_i\leq 5$ with the maximal connectivity of the inclusions, we may then apply either Lemma \ref{codimsmall} if $k_i\leq 4$ for  $i\in\{1, 2\}$ or the \hyperref[odd]{Odd Codimension Lemma} if $k_i=5$ for some $i\in\{1, 2\}$ to  the inclusions of $F_i$ in $F_{i+1}$ for $i\in\{1, 2\}$ to obtain integral cohomology ring  or homotopy equivalence  information about $F_1$ and $F_2$. 

Our goal is to show that $F_1$ is homotopy equivalent to a sphere or a real projective space. Since at least one of the $k_i$, $i\in\{1, 2\}$, is odd, we then see that the corresponding $F_i$ is homotopy equivalent to a sphere or a real projective space. If this is true for $F_1$, we are done. If not, the maximal connectivity of $F_1$ in $F_2$ then allows us to see  $F_1$ is a homotopy equivalent to a sphere or a real projective space, as desired.

 By Corollary \ref{chain!}, in Case 1, $F_0$ is maximally connected. In both Cases 1 and 2, the inclusion $F_0\hookrightarrow F_1$ is sufficiently connected to conclude that $F_0$ is also  homotopy equivalent to a sphere or a real projective space.
\end{proof}

\begin{lemma}\label{l2}
Let $M$ be a closed, positively curved manifold, and $F_i$ be closed, totally geodesic submanifolds of $M$ satisfying
$$F_0\subseteq F_1\subseteq F_2,$$
with $F_1\hookrightarrow F_2$ $m_1$-connected, and $m_0\geq 4$. Suppose   $k_0=2$, and $k_1=4$.
Then $F_0$ is homotopy equivalent to $S^{m_0}$, $\RP^{m_0}$, $\CP^{\frac{m_0}2}$, or a lens space.
\end{lemma}

\begin{proof}
We first claim that $\widetilde{F}_0$ is a $\Z$-cohomology sphere or complex projective space.  By Lemma \ref{connlifting} we can lift the chain of inclusions $F_0\underset{}{\subseteq}F_1 \underset{}{\subseteq}F_{2}$, to the chain of inclusions
		$\widetilde{F}_0\underset{}{\subseteq}\widetilde{F}_1 \underset{}{\subseteq}\widetilde{F}_{2},$
where the connectivity of $F_i\subseteq F_{i+1}$ is equal to the connectivity of $\widetilde{F}_i\subseteq \widetilde{F}_{i+1}$ for $i=0, 1.$
By the \hyperref[connlemma]{Connectedness Lemma}, we have that $\widetilde{F}_0\hookrightarrow \widetilde{F}_1$ is $(m_0-1)$-connected, with $(m_0-1)\geq 3$.   
Since $k_0=2$, applying the \hyperref[periodicity]{Periodicity Lemma} to $\widetilde{F}_0\hookrightarrow \widetilde{F}_1$ implies that there is an element $\alpha\in H^2(\widetilde{F}_1;\Z)$ such that $\smile \alpha:H^2(\widetilde{F}_1;\Z)\to H^4(\widetilde{F}_1;\Z)$ is an isomorphism. 
Since $F_1$ is $m_1$-connected in $F_2$ by hypothesis and $k_1=4$, we may apply the  \hyperref[periodicity]{Periodicity Lemma}
to conclude that $H^*(\widetilde{F}_2;\Z)$ 
is four-periodic. If the periodicity-inducing 
element in $H^*(\widetilde{F}_2;\Z)$ is 
trivial, then $\widetilde{F}_2$ is an 
integral cohomology sphere, and hence so are 
$\widetilde{F}_1$ and $\widetilde{F}_0$, respectively,  by the connectivity of their inclusions and Poincar\'e duality.

Suppose then that $x\in H^*
(\widetilde{F}_2;\Z)$ is non-trivial and 
induces four-periodicity.  
Since the inclusion $\iota:\widetilde{F}_1\hookrightarrow \widetilde{F}_2$ is $m_1$-connected, there is a non-trivial $y\in 
H^4(\widetilde{F}_1;\Z)$ so that 
$\iota^*(y)=x$. 
Since $y$ is non-zero, it follows that $y = \alpha\smile z$ with $z\in H^2(\widetilde{F}_1;\Z)$. 
Thus, $x\in H^4(\widetilde{F}_2;\Z)$ factors as a product of non-trivial elements in $H^2(\widetilde{F}_2;\Z)$.
Lemma \ref{factor} then implies that $H^*(\widetilde{F}_2;\Z)$ is $2$-periodic and thus $H^*(\widetilde{F}_1;\Z)$ is $2$-periodic. Thus, 
$\widetilde{F}_1$ is 
either an integral cohomology $S^{m_1}$ or $\CP^{m_1/2}$. The connectivity of the inclusion $\widetilde{F}_0\hookrightarrow \widetilde{F}_1$ and Poincar\'e duality give us the same result for  $\widetilde{F}_0$, thus proving the claim.

If $\widetilde{F}_0$ is an integral cohomology $S^{m_0}$, the 
result now follows from Lemma 
\ref{codim2cyclic} and Theorem 
\ref{universal}. 
Suppose then that $\widetilde{F}_0$  
has the integral cohomology of 
$\CP^{m_0/2}$. Since $k_0=2$ and  
$\widetilde{F}_1$ is also an integral 
cohomology complex projective space,  
one of $\euc(\widetilde{F}_0)$ or $\euc(\widetilde{F}_1)$ is 
odd. Since $\Z_2$ cannot act freely 
on one of $\widetilde{F}_0$ or 
$\widetilde{F}_1$,  
 both $F_0$ and 
$F_1$ must have
trivial fundamental group. We may then apply Theorem 
\ref{universal} to see that $F_0$ is 
homotopy equivalent to 
$\CP^{m_0/2}$.

\end{proof}

\begin{lemma}\label{l3}
Let $M$ be a closed, positively curved manifold, and $F_i$ be closed, totally geodesics submanifolds of $M$ satisfying
$$F_0\subseteq F_1\subseteq F_2,$$
with $F_1\hookrightarrow F_2$ $m_1$-connected, and $m_0\geq 4$. Suppose  either $k_0=2$ and $4\leq k_1\leq 7$, or  $k_0=3$, and $6\leq k_1\leq 7$.
Then $F_0$ has the $\Z_2$-cohomology of $S^{m_0}$, $\RP^{m_0}$, $\CP^{\frac{m_0}2}$, or a lens space.
\end{lemma}

\begin{proof}
    We note that Lemma \ref{l2} gives us the desired result when $k_1=4$, so we only need consider the cases where $k_1\in \{5, 6, 7\}$.
Since the inclusion $F_1\hookrightarrow F_2$ is $m_1$-connected and $m_1\geq 8$, we apply Part 1 of Lemma \ref{chain2} or Part 2 of Lemma \ref{chain3} to 
obtain that $F_1$ has the $\Z_2$-cohomology of $S^{m_0}$, $\RP^{m_0}$, $\CP^{\frac{m_0}2}$, or a lens space. For all of these cases but one, namely when $F_1$ has the $\Z_2$-cohomology of a lens space, we apply Theorem \ref{Bredon} to obtain that $F_0$ is as in the conclusion of Lemma \ref{l3}.

In the case where $F_1$ is a $\Z_2$-cohomology lens space,  $m_1$ must be odd. If $k_0=2$, the result follows from Proposition \ref{codimsmall}. It remains to consider the case when $k_0=3$. 
Since $m_0$ must be even,  $\pi_1(F_0)$ is either 0 or $\Z_2$ by Synge's theorem. It follows from the \hyperref[connlemma]{Connectedness Lemma} that $\pi_1(F_1)$ is either $0$ or $\Z_2$, and by Theorem \ref{Su}, $F_1$ is either a $\Z_2$-cohomology sphere or real projective space. The result  follows now from Theorem \ref{Bredon}.
\end{proof}

\section{The Proof of Theorem A}\label{5}
In this section we prove Theorem \hyperref[TA]{A}, which we restate with more detail in Part 2 for the convenience of the reader, see Observation \ref{quotients}.

\begin{theorem}\label{Atech} Let $M^n$ be a closed, positively curved manifold, and assume $\Z_2^r$ acts effectively by isometries on $M$ with non-empty fixed-point set, with $n\geq 7$. Let $F^m$ denote the fixed-point set component of $\Z_2^r$ containing $x$ and suppose that $m\geq 4$. Suppose 
		\[r\geq\frac{n}{6}+1.\]		
\noindent Then one of the following holds:

\begin{enumerate}
\item $F^m$ is homotopy equivalent to one of  $S^m$, $\RP^m$, $\CP^{\frac{m}2}$, or a lens space; or
\item   $F^m$ has the integral cohomology ring of   $S^m/\Gamma$, $(S^3/\Gamma) \times \HP^{\frac{m-3}{4}}$, $N_j^{}/\Gamma$ with $m\equiv  3 \pmod 4$, $\CP^{\frac{m}{2}}/\Delta$, $(S^2/\Delta)\times \HP^{\frac{m-2}{4}}$,  $\HP^{\frac{m}{4}}$, or $N^m_j/\Gamma$, $E^{m}_{\ell}/\Delta$, $\ell\geq 2$, with  $m\equiv 2 \pmod 4$ and $j$ odd, where $\Gamma$ is some finite group acting freely on the corresponding manifold  and $\Delta$ is  a subgroup of $\Z_2$.
\end{enumerate}
\end{theorem}

\begin{remark} When $\dim(F_0)=3$, we obtain  that $F_0^3$ is homotopy equivalent to a sphere, real projective space or lens space, in all but one case. The condition under which this occurs is given in Proposition \ref{dim3}. 
 In Example 7.3 of \cite{KKSS}, they describe $\Z_2^r$-actions with a fixed point on $n$-manifolds satisfying $r\geq n/6+1$  for which 
$3$-dimensional spherical space-forms with non-cyclic fundamental group show up as fixed-point set components of the action. 
\end{remark}

\begin{proposition}\label{dim3}
Let $M^n$ be a closed, positively curved manifold, and assume $\Z_2^r$ acts effectively by isometries on $M$ with non-empty fixed-point set, with $n\geq 7$. Let $F^3_x$ denote the fixed-point set component of $\Z_2^r$ containing $x$ of smallest dimension.  Suppose
		\[r\geq\frac{n}{6}+1.\]		
  Then $F^3$ is  homotopy equivalent to $S^3$, $\RP^3$, or a lens space, provided 
 for the chain of  closed, submanifolds, $F_i$ corresponding to fixed-point sets of $\Z_2^{r-i}$ the case $k_0=k_1=k_2=4$ does not occur. If $k_0=k_1=k_2=4$, then $F^3$ is homotopy equivalent to a spherical space form.
\end{proposition}

\begin{proof} If $k_0=k_1=k_2=4$, since $F^3$ is positively curved, it follows by work of Hamilton \cite{Ha} and Wolf \cite{Wo} that $F^3$ is homotopy equivalent to a spherical space form. 

We now suppose that the case $k_0=k_1=k_2=4$ does not occur.    Note first that if $k_0=1$ or $2$,  the result follows from  Parts 1 and 2 of Lemma \ref{chain2}, respectively.
By Lemma \ref{bounds}, we have that if $k_0\geq 2$, then $2\leq k_i\leq 5$ for $0\leq i\leq 2$. It then follows that $k_1\leq 2k_0-1$ for $k_0\geq 3$. In particular, we may then apply Lemma \ref{chain2} to obtain the result for $k_0=3$.

It remains to consider the cases where $k_0=4$ or $5$. 
When $k_0=4$, either $k_1=4$ or $5$. In both cases, we have that  
$k_2=5$. Moreover, using a similar argument as in Lemma \ref{bounds}, it follows that $k_3=5$ and thus $k_3\leq 2k_2-1$.  By Corollary \ref{chain!}, $F_2$ is maximally connected in $F_3$. It then from Part 1 of Lemma \ref{chain2} that $F_3$ is homotopy equivalent to a sphere or real projective space. Since $F_i$ is maximally connected in $F_{i+1}$, $0\leq i\leq 2$,  the result follows.
Likewise, when $k_0=5$, we may then apply Part 1 of Lemma \ref{chain2} to obtain that $F_1$ is homotopy equivalent to a sphere or real projective space,  and again since $F_0$ is maximally connected in $F_1$, the result follows.

\end{proof}

We are now ready to prove Theorem \hyperref[TA]{A}.
\begin{proof}[Proof of Theorem A]
 Let $M^n$ be a closed, positively curved manifold, and assume $\Z_2^r$ acts effectively by isometries on $M$ with fixed-point $p$ and fixed-point set component $F$ at $p$, with $n\geq7$  and $r\geq\frac{n}{6}+1$.
By Proposition \ref{chain} we have a chain of inclusions
\[F_0\subseteq F_1 \subseteq \cdots \subseteq F_j \subseteq F_{j+1}\subseteq\cdots {\subseteq} \,F_{r-1}{\subseteq}\, M,\]
\noindent where the effective kernel of the induced  $\Z_2^r$-action on $F_j$ is 
 a corank-$j$ subgroup of $\Z_2^r$, $\dim\(F_j\)=m_j$, $k_j = \codim(F_j\subset F_{j+1})$,  and $k_j\leq k_{j+1}$ for all $0\leq j\leq r-1$.

Suppose $k_j=1$ for some $0\leq j\leq r-1$. Since the $k_j$ constitute a non-decreasing sequence of positive integers, we know that $k_0=1$.  
It follows from Part 1 of Lemma \ref{chain2} 
that $F_0$ is homotopy equivalent to either $S^{m_0}$ or $\RP^{m_0}$ and hence Theorem \hyperref[TA]{A} holds.

We now assume that $k_0\geq 2$. By Lemma \ref{bounds}, it follows that $2\leq k_i\leq 5$ for $0\leq i\leq 2$. Recall that we have two cases: Case 1, where $m_0\geq k_0$ and Case 2, where $m_0<k_0$. 
\vspace{.1cm}

\noindent{\bf Case 1,  where ${\bf m_0\geq k_0}$}: 
We first assume that $k_0=2$.  If $2\leq k_1\leq 3$, 
then by Corollary \ref{chain!}, $F_0$ is maximally connected in $F_1$ and the desired result follows. We suppose then that $k_1\in\{4, 5\}$. If $k_1=4$, then $k_2 \le 2k_1-1$ since $k_2 \le 5$, and hence $F_1 \hookrightarrow F_2$ is $m_1$-connected by Corollary \ref{chain!}.
We then may apply Lemma \ref{l2} to obtain the result. If instead $k_1=5$, using a similar argument as in Lemma \ref{bounds}, it follows that $k_3=5$ and thus $k_3\leq 2k_2-1$.
Since $F_0$ is at least $(m_0-1)$-connected by the \hyperref[connlemma]{Connectedness Lemma},
then we may then apply Lemma \ref{l1} to obtain the result.

We now assume that $k_0\geq 3$. Since $k_1\leq 5$, it follows that $k_1\leq 2k_0-1$ and by Corollary \ref{chain!}, we have $F_0$ is maximally-connected in $F_1$. For $3\leq k_0\leq 5$, the result then follows by Part 1 or Part 3 of Proposition \ref{codimsmall}, depending on the parity of $k_0$.
\vspace{.1cm}

\noindent{\bf Case 2,  where ${\bf m_0< k_0}$}: 
Here $k_0=5$ since $m_0 \ge 4$. It is straightforward to verify that in this case the hypotheses of Case a of Lemma \ref{l1} are satisfied and thus we obtain the desired result.

This concludes the proof of Theorem \hyperref[TA]{A}.
\end{proof}

\section{The proof of Theorem B}\label{6}
We now give the proof of Theorem \hyperref[TB]{B}, which is immediate once we prove the following theorem.

\begin{theorem}\label{Btech} 
Let $M^n$ be a closed, positively curved Riemannian manifold, and assume $\Z_2^r$ acts effectively by isometries on $M$ with a fixed point, $x\in M$, with $n\geq 9$. Suppose that 
		\[r\geq\frac{n}8+2.\]		
\noindent Let $F^m_x$ be a connected component of the fixed-point set of $\Z_2^r$ containing $x$.  Then at least one of the following holds:
\begin{enumerate}
\item $F^m$ is homotopy equivalent to one of  $S^m$, $\RP^m$, $\CP^{\frac{m}2}$, or a lens space;
\item   The universal cover of $F^m$, $\widetilde{F}^m$, has the integral cohomology ring of   $S^m$, $S^3 \times \HP^{\frac{m-3}{4}}$, $N_j^{m}$ with $m\equiv 2, 3 \pmod 4$, $\CP^{\frac{m}{2}}$, $S^2\times \HP^{\frac{m-2}{4}}$, $E^{m}_{\ell}$, $\ell\geq 2$ with  $m\equiv 2 \pmod 4$, or $\HP^{\frac{m}{4}}$; or
\item $F^m$ has the $\Z_2$-cohomology ring of $S^{m}$, $\CP^{m/2}$, $\RP^m$, 
 $\HP^{m/4}$, $S^2\times \HP^{\frac{m-2}4}$, $\RP^2\times \HP^{\frac{m-2}4}$, $S^3/\Z_k\times\HP^{\frac{m-3}4}$ with $k\in\{1,2,4\}$, $S^{m}/Q_i$ for $i\in\{8,16\}$ and $m_j\equiv 3\pmod{4}$, or a lens space.
 \end{enumerate}

\end{theorem}

\begin{proof}[Proof of Theorem 6.1] 
Let $M^n$ be a closed, positively curved manifold, and assume $\Z_2^r$ acts effectively by isometries on $M$ with fixed-point $x$ and fixed-point set component $F_0$ containing $x$. Suppose further that $n\geq 9$ and $r\geq\frac{n}{8}+2$.
As in the proof of Theorem \hyperref[TA]{A} we have a chain of inclusions
\[F_0\subseteq F_1 \subseteq \cdots \subseteq F_j \subseteq F_{j+1}\subseteq\cdots {\subseteq} \,F_{r-1}{\subseteq}\, M,\]
\noindent where the effective kernel of the induced  $\Z_2^r$-action on $F_j$ is 
 a corank-$j$ subgroup of $\Z_2^r$, $\dim\(F_j\)=m_j$, $k_j = \codim(F_j\subset F_{j+1})$,  and $k_j\leq k_{j+1}$ for all $0\leq j\leq r-1$.

As we saw in the proof of Theorem \hyperref[TA]{A}, the result holds when $k_0=1$, so we assume that $k_j\geq 2$ for all $0\leq j\leq r-1$. 
In the cases where $0\leq m_0\leq 2$, since $F_0$ is a closed totally geodesic submanifold, Theorem B holds by the classification of closed manifolds of dimensions $0$ and $1$, and by the Gauss-Bonnet Theorem in dimension $2$. Note that if $m_0=3$, then Theorem \hyperref[Btech]{6.1} follows from Corollary \ref{cors3}, so we assume $m_0\geq 4$.

By Lemma \ref{bounds}, we have $2\leq k_j\leq 7$ for $0\leq j\leq 3$. 
As in the proof of  Theorem \hyperref[TA]{A}, we have two cases: Case 1, where $m_0\geq k_0$, and Case 2, where $m_0<k_0$. 
\vspace{.1cm}

\noindent{\bf Case 1,  where ${\bf m_0\geq k_0}$}:
We break Case 1 into two further subcases,   where $k_1\leq 2k_0-1$, and  where $k_1\geq 2k_0$.  We first assume that $k_1\leq 2k_0-1$, and note that we may apply Lemma \ref{chain2} and Lemma \ref{chain3} to obtain the result. If instead we assume that $k_1\geq 2k_0$, since $k_1\leq 7$, we need only consider $2\leq k_0\leq 3$. When $k_0=2$, then $4\leq k_1\leq 7$ and for $k_0=3$, $6\leq k_1\leq 7$.  In both cases $k_2 \le 2k_1-1$ since $k_2 \le 7$, and hence $F_1 \hookrightarrow F_2$ is $m_1$-connected by Corollary \ref{chain!}.
We may then apply Lemma \ref{l3} to each of these cases to obtain the result.
\vspace{.1cm}

\noindent{\bf Case 2,  where ${\bf m_0< k_0}$}:
Here, $5\leq k_0\leq 7$ since $m_0 \ge 4$. Since $ 2 \le k_j \le 7$ for $j=1, 2$, and $k_1 \le k_2$, we obtain that $k_1\leq 2k_0-1$ and $k_2\leq 2k_1-1$.  It follows by Corollary \ref{chain!} that $F_i$ is  maximally connected in $F_{i+1}$ for $i\in \{0, 1\}$.  Since $m_1\geq k_1$, we may then apply Lemma \ref{chain2} and Lemma \ref{chain3} to see that  $F_1$ is as in the conclusion of Theorem \hyperref[Btech]{6.1}. Since $F_0$ is $m_0$-connected in $F_1$, the result follows.

This concludes the proof of Theorem \hyperref[Btech]{6.1}.
\end{proof}


\end{document}